\documentclass[onefignum,onetabnum]{siamart171218}
\usepackage{lipsum}
\usepackage{amsfonts}
\usepackage{graphicx}
\usepackage{epstopdf}
\usepackage{algorithmic}

\usepackage{amsmath,amssymb}
\ifpdf
  \DeclareGraphicsExtensions{.eps,.pdf,.png,.jpg}
\else
  \DeclareGraphicsExtensions{.eps}
\fi

\newsiamremark{remark}{Remark}
\newsiamremark{hypothesis}{Hypothesis}
\crefname{hypothesis}{Hypothesis}{Hypotheses}
\newsiamthm{claim}{Claim}
\newsiamremark{example}{Example}

\usepackage{enumerate}

\usepackage[mathscr]{euscript}
\usepackage{url}
\usepackage[all]{xy}
\usepackage{booktabs}
\usepackage{algorithmic}
\usepackage[caption=false]{subfig}
\usepackage{cancel}
\usepackage{cleveref}

\usepackage{tikz}
\usetikzlibrary{shapes}
\usetikzlibrary{arrows}
\usetikzlibrary{matrix}

\newcommand{\tens}[1]{\boldsymbol{\mathcal{#1}}}
\newcommand{\tenselem}[1]{\mathcal{#1}}

\newcommand{\matr}[1]{\boldsymbol{#1}}

\newcommand{\vect}[1]{\boldsymbol{#1}}

\newcommand{\set}[1]{\mathscr{#1}}
\newcommand{\T}{{\sf T}}        
\renewcommand{\H}{{\sf H}}      

\makeatletter
\newcommand{\rank}[1]{\mathop{\operator@font rank}\{#1\}}
\newcommand{\colrank}[1]{\mathop{\operator@font colrank}\{#1\}}
\newcommand{\krank}[1]{\mathop{\operator@font krank}\{#1\}}
\newcommand{\trace}[1]{\mathop{\operator@font tr}\{#1\}}
\newcommand{\Diag}[1]{\mathop{\operator@font Diag}\{#1\}}    
\newcommand{\diag}[1]{\mathop{\operator@font diag}\{#1\}}    
\newcommand{\Span}[1]{\mathop{\operator@font Span}\{#1\}}    
\newcommand{\argmin}{\mathop{\operator@font argmin}}

\newcommand{\offdiag}[1]{\mathop{\operator@font offdiag}\{#1\}}    
\newcommand{\Proj}[2]{\mathop{\operator@font Proj_{#1}}{#2}}
\newcommand{\ProjGrad}[2]{\mathop{{\operator@font grad} }#1(#2)}
\newcommand{\Hess}[2]{\mathrm{Hess}_{#2}{#1}}
\newcommand{\HessAppl}[3]{\mathrm{Hess}_{#2}{#1} [#3]}
\makeatother

\newcommand{\eqdef}{\stackrel{\sf def}{=}}
\newcommand{\RR}{\mathbb{R}}
\newcommand{\CC}{\mathbb{C}}
\newcommand{\NN}{\mathbb{N}}

\newcommand{\UN}[1]{\set{U}_{#1}}
\newcommand{\tUN}[1]{\widetilde{\set{U}}_{#1}}

\newcommand{\Gmat}[3]{\matr{G}^{(#1,#2,#3)}}

\newcommand{\contr}[1]{\mathop{\bullet_{#1}}}   
\newcommand{\TangV}[2]{\mathcal{V}_{#1}{#2}}
\newcommand{\TangH}[2]{\mathcal{H}_{#1}{#2}}
\newcommand{\HessPS}[3]{\mathfrak{D}^{(#1,#2)}_{#3}}

\newcommand{\Tang}[2]{\mathbf{T}_{#1}{#2}}
\newcommand{\TangBundle}[1]{\mathbf{T}{#1}}

\newcommand{\deriv}[2]{{\nabla}_{#2}  {#1}}   
\newcommand{\R}{\Re}
\newcommand{\I}{\Im}
\newcommand{\Utwo}{\Psi}
\newcommand{\hij}[3]{h_{(#1,#2), #3}}
\newcommand{\Gamij}[3]{\matr{\Gamma}^{(#1,#2,#3)}}
\newcommand{\RInner}[2]{\left\langle #1, #2\right\rangle_{\Re}}   

\newcommand{\GamTens}{\tens{F}}

\makeatletter
\def\bbordermatrix#1{\begingroup \m@th
  \@tempdima 4.75\p@
  \setbox\z@\vbox{%
    \def\cr{\crcr\noalign{\kern2\p@\global\let\cr\endline}}%
    \ialign{$##$\hfil\kern2\p@\kern\@tempdima&\thinspace\hfil$##$\hfil
      &&\quad\hfil$##$\hfil\crcr
      \omit\strut\hfil\crcr\noalign{\kern-\baselineskip}%
      #1\crcr\omit\strut\cr}}%
  \setbox\tw@\vbox{\unvcopy\z@\global\setbox\@ne\lastbox}%
  \setbox\tw@\hbox{\unhbox\@ne\unskip\global\setbox\@ne\lastbox}%
  \setbox\tw@\hbox{$\kern\wd\@ne\kern-\@tempdima\left[\kern-\wd\@ne
    \global\setbox\@ne\vbox{\box\@ne\kern2\p@}%
    \vcenter{\kern-\ht\@ne\unvbox\z@\kern-\baselineskip}\,\right]$}%
  \null\;\vbox{\kern\ht\@ne\box\tw@}\endgroup}
\makeatother

\def\bsmallmatrix#1{\left[\begin{smallmatrix}#1\end{smallmatrix}\right]}

\newcommand{\ue}{\mathrm{e}}
\newcommand{\ui}{i}
\newcommand{\uI}{\matr{I}}

\ifpdf
\hypersetup{
  pdftitle={On convergence of matrix and tensor approximate diagonalization algorithms},
  pdfauthor={K. Usevich, J. Li, and P. Comon}
}
\fi

\title{Approximate matrix and tensor diagonalization \\  by unitary transformations: \\ convergence of Jacobi-type algorithms\thanks{Submitted to the editors DATE.
\funding{This work was funded in part by the ERC project ``DECODA'' no.320594, in the frame of the European program FP7/2007-2013,  by the National Natural Science Foundation of China (No.11601371), 
and  by  the   Agence Nationale de Recherche (ANR grant LeaFleT, ANR-19-CE23-0021).}}}
\author{Konstantin Usevich\thanks{Universit\'{e} de Lorraine, CNRS, CRAN, Nancy, France
  (\email{konstantin.usevich@univ-lorraine.fr}).}
\and Jianze Li\thanks{Shenzhen Research Institute of Big Data, The Chinese University of Hong Kong, Shenzhen, China 
  (\email{lijianze@gmail.com}).}
\and Pierre Comon\thanks{Univ. Grenoble Alpes, CNRS, Grenoble INP, GIPSA-Lab, France 
  (\email{pierre.comon@gipsa-lab.fr}).}}

\begin{document}
\maketitle

\begin{abstract}
We propose a gradient-based Jacobi algorithm for a class of maximization problems on the unitary group, with a focus on approximate diagonalization of complex matrices and tensors by unitary transformations.
We provide weak convergence results, and prove local linear convergence of this algorithm.
The convergence results also apply to the case of real-valued tensors.
\end{abstract}



\section{Introduction}
In this paper, we consider the following optimization problem
\begin{equation}\label{eq-gene-problem}
\matr{U}_{\ast} = \arg\max_{\matr{U} \in \UN{n}} f(\matr{U}),
\end{equation}
where $\UN{n}$ is the unitary group and $f:\UN{n}\rightarrow\RR$
is  real differentiable. 
An important class of such  problems stems from approximate matrix and tensor diagonalization  in numerical linear algebra \cite{BunsBM93:simax},  signal processing \cite{comon2010handbook} and machine learning \cite{Anan14:latent}.

Jacobi-type algorithms are widely used for maximization of these cost functions.
Inspired by the classic Jacobi algorithm \cite{GoluV96:jhu} for the symmetric eigenvalue problem, they  proceed by successive Givens  rotations that update only a pair of columns of $\matr{U}$.
The popularity of these approaches is explained by low computational complexity of the updates.
Despite their popularity, their convergence  has not yet been studied thoroughly, except the   case of matrices
 \cite{GoluV96:jhu}  and a pair of commuting matrices \cite{BunsBM93:simax}.

For tensor problems in the real-valued case (orthogonal group), a  gradient-based Jacobi-type algorithm  was proposed in \cite{IshtAV13:simax}, and its weak convergence\footnote{every accumulation point is a stationary point.} was proved\footnote{The algorithm \cite{IshtAV13:simax}  was proposed for a particular problem of Tucker approximation, but the convergence result of \cite{IshtAV13:simax}  are valid  for arbitrary smooth functions, see discussion in \cite{LUC2018}.}
In \cite{LUC2018},  its global (single-point) convergence\footnote{i.e., for any starting point, the iterations converge to a single limit point. Note that global convergence does not imply convergence to a global minimum; also, the notion of ``global convergence'' often has a different meaning in the numerical linear algebra community \cite{HariB19:complex}.}   for joint real 3rd order tensor or matrix diagonalization was proved. The proof in \cite{LUC2018}  
based on the \L{}ojasiewicz gradient inequality,  a  popular tool for studying convergence properties of nonlinear optimization algorithms \cite{AbsMA05:sjo,Lage07:phd,SU15:pro,AttoBS13:convergence}, including various tensor approximation problems \cite{Usch15:pjo,Hu2018}.

In this paper, we address the complex-valued case \eqref{eq-gene-problem}, and focus on  tensor and matrix approximate diagonalization problems.
Unlike the real case, where the Givens rotations are univariate (``line-search'' type), in the complex case the updates correspond to maximization on a sphere (similar in spirit to subspace methods).
The main contributions of the paper are: (i) we generalize the algorithm of \cite{IshtAV13:simax} to the  complex case,  prove its weak  convergence, and find global rates of convergence based on the results of \cite{boumal2016global}; (ii)  we show that the local convergence can be studied by combining the tools of \L{}ojasiewicz gradient inequality, geodesic convexity and  recent results on \L{}ojasiewicz exponent for Morse-Bott functions.
In particular, local linear convergence holds for local maxima satisfying second order regularity conditions. 
One of the motivations for this work was that the case of the unitary group is not  common in the optimization literature, unlike the orthogonal group or other matrix manifolds \cite{Absil08:Optimization}.

The structure of the paper is as follows. 
In \Cref{sec:previous},  we recall the  cost functions of  interest, the principle of Jacobi-type algorithms,  present the gradient-based  algorithm and a summary of main results. 
\Cref{sec:unitary_geom},  contains all necessary facts for differentiation on the unitary group. 
\Cref{sec:deriv_forms} contains  expressions for the  first- and second-order  derivatives, as well as expressions for Jacobi rotations for  cost functions of interest. 
In \Cref{sec:weak}, we present the results on weak convergence and global convergence rates. 
The  results  of \cite{boumal2016global} are summarized  in the same section.
In \Cref{sec:Lojasiewicz}, we recall  results  based on  \L{}ojasiewicz gradient inequality, and  facts on Morse-Bott functions. 
\Cref{sec-main-results} contains main results  and  lemmas.

\section{Background, problem statement, and summary of results}\label{sec:previous}
\subsection{Main notation} 
For an $\matr{X} \in \CC^{m\times n}$, we denote by $\matr{X}^{*}$ its elementwise conjugate, and by $\matr{X}^{\T}$, $\matr{X}^{\H}$ its transpose and Hermitian transpose.
We   use the following notation for the real and imaginary parts  $\matr{X} =\matr{X}^{\R} + \ui \matr{X}^{\I}$  of matrices, and $\R(z)$, $\I(z)$  for  $z\in \CC$. 
Let $\mathbb{T} \subset \CC$ be the unit circle, and $\UN{n}\subset\CC^{n\times n}$ be the unitary group. 

In this paper, we make no distinction between tensors and multi-way arrays; for simplicity, we consider only fully contravariant tensors \cite{Mccu87}. 
For a tensor or a matrix $\tens{A} \in \CC^{n\times \cdots \times n}$, we denote by $\diag{\tens{A}} \in \CC^{n}$ the vector of  all the diagonal elements $\tenselem{A}_{ii\cdots i}$  and by $\trace{\tens{A}}$ the sum of the diagonal elements. 
We denote by $\|\cdot\|$ the Frobenius norm of a tensor/matrix,
or the Euclidean norm of a vector. 
For a $d$-th order tensor $\tens{A}\in \CC^{n\times \cdots \times n}$ its contraction on the $k$th index  with  $\vect{v} \in \CC^n$ (resp.  $\matr{M} \in \CC^{m\times n}$)  is 
\[
(\tens{A} \contr{k} \vect{v}) _{i_1..{\cancel{i_k}}..i_d}  \eqdef \sum_{{j}=1}^{{n}} \tens{A}_{i_1.. i_{k\text{-}1} {j}i_{k+1}.. i_d}  {v}_{{j}},\;
(\tens{A} \contr{k} \matr{M}) _{i_1..i_d}  \eqdef \sum_{{j}=1}^{{n}} \tens{A}_{i_1.. i_{k\text{-}1} {j}i_{k+1}.. i_d}  {M}_{i_k,{j}}.
\]
By writing multiple contractions $\tens{A}\contr{k_1} \vect{v_1} \ldots \contr{k_\ell} \vect{v}_{\ell}$ we assume that they are performed simultaneously, i.e.,  the indexing of the tensor does not change before contractions are complete.
For a matrix $\matr{S} \in \CC^{n\times n}$, we will also denote the double contraction as
\[
(\tens{A} \contr{k,\ell} \matr{S}) _{i_1..{\cancel{i_k}}..{\cancel{i_{\ell}}}...i_d}  \eqdef \sum_{{j,s}=1}^{{n,n}} \tens{A}_{i_1\ldots i_{k-1} {j}i_{k+1}\ldots  i_{\ell-1} {s}i_{\ell+1} \ldots i_d}  {S}_{j,s} \\
\]
For a matrix $\matr{U} \in \CC^{n\times n}$, we will denote its columns as  
$\matr{U}  = \begin{bmatrix} \vect{u}_1 & \cdots  & \vect{u}_n \end{bmatrix}$. 

\subsection{Motivation}\label{sec-cost-fn}
This paper is motivated by following maximization problems: 
\begin{enumerate}[(i)]
\item joint \emph{approximate Hermitian diagonalization  of matrices} 
$\matr{A}^{(\ell)} \in \CC^{n\times n}, 1\!\leq\!\ell\!\leq\!L$:
\begin{equation}\label{eq-cost-jade}
f(\matr{U}) =  \sum\limits_{\ell=1}^{L}\|\diag{\matr{U}^{\H}\matr{A}^{(\ell)}\matr{U}}\|^2 = \sum\limits_{\ell=1}^{L} \sum\limits_{p=1}^{n}|\vect{u}^{\H}_p\matr{A}^{(\ell)}\vect{u}_p|^2;
\end{equation}
\item \emph{approximate  diagonalization of a 3rd order tensor} 
 $\tens{A}\in\CC^{n\times n\times n}$: 
\begin{equation}\label{eq-cost-3-tensor}
f(\matr{U}) =  \|\diag{\tens{A} \contr{1} \matr{U}^{\H} \contr{2} \matr{U}^{\T}\contr{3} \matr{U}^{\T}}\|^2 = 
\sum\limits_{p=1}^{n}|\tens{A} \contr{1} \vect{u}^*_p \contr{2} \vect{u}_p \contr{3} \vect{u}_p|^2;
\end{equation}

\item \emph{approximate   diagonalization of a  4th order tensor} $\tens{B}\in\CC^{n\times n\times n\times n}$ satisfying  a
 Hermitian  symmetry condition  $\tenselem{B}_{ijkl}=\tenselem{B}^{*}_{klij}$  for  any $1\leq i,j,k,l\leq n$:
\begin{equation}\label{cost-4-order}
f(\matr{U}) = \trace{\tens{B} \contr{1} \matr{U}^{\H} \contr{2} \matr{U}^{\H}\contr{3} \matr{U}^{\T}\contr{4} \matr{U}^{\T}} =
\sum\limits_{p=1}^{n}\tens{B} \contr{1} \vect{u}^*_p \contr{2} \vect{u}^*_p \contr{3} \vect{u}_p \contr{4} \vect{u}_p.
\end{equation}
\end{enumerate}
Such maximization problems appear in blind source separation \cite{comon2010handbook} in the context of: 
\begin{enumerate}[(i)]
\item  joint diagonalization of covariance matrices \cite{cardoso1993blind,cardoso1996jacobi};
\item diagonalization of  the cumulant tensor \cite{de1997signal} with $\tenselem{A}_{ijk}=\text{Cum}(v_i,v_j^{*},v_k^{*})$;
\item  diagonalization  of  the cumulant tensor \cite{comon2001source} $\tenselem{A}_{ijkl}=\text{Cum}(v_i,v_j,v_k^{*},v_l^{*})$    of a complex random vector $\vect{v}$, which may itself stem from a Fourier transform \cite{EmilCL98:ieeesp}.
\end{enumerate}
\begin{remark}\label{rem:approx_diag}
Due to invariance of  $\|\cdot\|$ to unitary transformations, maximizing  \eqref{eq-cost-jade} or \eqref{eq-cost-3-tensor} is equivalent to minimizing  sums of squares of  the off-diagonal elements of the rotated tensors/matrices, hence the name
 ``approximate diagonalization''.
For example, in the single matrix case (i.e., \eqref{eq-cost-jade}  and $L=1$), we can equivalently minimize the squared norm of the off-diagonal elements (so called off-norm) 
\begin{equation}\label{eq-off-norm}
\|\text{off} (\matr{U}^{\H}\matr{A}\matr{U})\|^2 = \|\matr{A}\|^2 - \|\diag{\matr{U}^{\H}\matr{A}\matr{U}}\|^2,
\end{equation}
which is typically done in the numerical linear algebra community \cite{GoluV96:jhu}.
\end{remark}
In this paper, we consider a class of  functions that generalizes\footnote{It is easy to see that \eqref{cost-fn-general} generalizes \eqref{eq-cost-jade} (for $d_1 = \cdots = d_L = 2, t_1 = \cdots = t_L = 1$) and  \eqref{eq-cost-3-tensor} (for  $L = 1, d_1 = 3, t_1 = 1$).
} \eqref{eq-cost-jade}--\eqref{cost-4-order}.
For a set of tensors $\tens{A}^{(1)}, \ldots, \tens{A}^{(L)}$ of orders $d_1,\ldots, d_L$ (potentially different), integers $t_\ell$, $0 \le t_\ell \le d_\ell$, and  $\alpha_\ell \in \RR$ (possibly negative), we define the  cost function as
\begin{equation}\label{cost-fn-general}
f(\matr{U}) = \sum\limits_{\ell=1}^{L}
 \alpha_\ell \|\diag{\tens{A}^{(\ell)} \contr{1} \matr{U}^{\H}\cdots \contr{t_\ell} \matr{U}^{\H} \contr{t_\ell+1} \matr{U}^{\T} \cdots \contr{d_\ell} \matr{U}^{\T}}\|^2,
\end{equation}
i.e., a conjugate transformation is applied $t_\ell$ times and a non-conjugate $d_\ell - t_\ell$ times.
If all $\alpha_k > 0$, maximization of \eqref{cost-fn-general} can be viewed as joint diagonalization of several tensors (as in \Cref{rem:approx_diag}); the general case of negative $\alpha_k$ allows for more flexibility.
Also, \eqref{cost-fn-general} includes symmetric diagonalization problems (without conjugations), e.g.,
\[
 f(\matr{U}) = \sum\limits_{\ell=1}^{L}\|\diag{\matr{U}^{\T}\matr{A}^{(\ell)}\matr{U}}\|^2,\quad \text{or }
f(\matr{U}) =  \|\diag{\tens{A} \contr{1} \matr{U}^{\T} \contr{2} \matr{U}^{\T}\contr{3} \matr{U}^{\T}}\|^2;
\]
It can be shown that $f$ admits representation \eqref{cost-fn-general} (with $d=\max(d_1,\ldots,d_L)$) if and only if there exists a $2d$-th order tensor
$\tens{B}$ that is Hermitian \cite{NieY19:Hermitian}, i.e., 
\begin{equation}\label{eq:hermitian}
\tenselem{B}_{i_1\cdots i_d j_1 \cdots j_d} = \tenselem{B}^{*}_{ j_1 \cdots j_d i_1\cdots i_d}
\end{equation}
  such that $f$ has a representation which generalizes \eqref{cost-4-order}:
\begin{equation}\label{cost-fn-general-trace}
f(\matr{U}) = \trace{\tens{B} \contr{1} \matr{U}^{\H} \cdots \contr{d} \matr{U}^{\H}\contr{d+1} \matr{U}^{\T} \cdots \contr{2d} \matr{U}^{\T}}.
\end{equation}
The equivalence between \eqref{cost-fn-general} and \eqref{cost-fn-general-trace} is analogous to the spectral theorem for Hermitian matrices; a proof can be found  in  \Cref{sec:deriv_forms}  (see also \cite[Prop. 3.5]{JianL16:characterizing}).

\subsection{Jacobi-type methods}
Fix an index pair $(i,j)$ that satisfies $1\leq i<j\leq n$.
Then, for a matrix $\matr{\Utwo} \in \UN{2}$, we define the  \emph{plane transformation} in $\UN{n}$ as:
\begin{equation}\label{eq:givens_trans}
\Gmat{i}{j}{\matr{\Utwo}} = {\small
\bbordermatrix{
&  & & i& & j && \cr
&1 &       & & & &&\cr
 & &\ddots & & & & \matr{0} &\cr
 i && & \Utwo_{1,1} & & \Utwo_{1,2} && \cr
 & & & & \ddots & &&\cr
j & & & \Utwo_{2,1} & & \Utwo_{2,2} && \cr
&  & \matr{0} & & & & \ddots &\cr
& &       & & & &&1}, }
\end{equation}
which coincides with $\matr{I}_n$ except  the positions $(i,i), (i,j), (j,i), (j,j)$. 
The set of matrices $\Gmat{i}{j}{\matr{\Utwo}}$ is a subgroup of $\UN{n}$ that is  canonically  isomorphic to $\UN{2}$.

Jacobi-type methods aim at maximizing the cost function  by applying successive plane  transformations.
The  iterations $\{\matr{U}_k\}$  are  generated multiplicatively
\[
\matr{U}_k = \matr{U}_{k-1} \Gmat{i_k}{j_k}{\matr{\Utwo}_k},
\]
where  $(i_k,j_k)$ is chosen according to a certain rule, and $\matr{\Utwo}_k$ maximizes the restriction $\hij{i_k}{j_k}{\matr{U}_{k-1}}$ of $f$ defined as 
\begin{equation}\label{eq-func-h-new}
  \begin{split}
\hij{i}{j}{\matr{U}}:\quad&\UN{2} \longrightarrow \RR\\
& {\matr{\Utwo}}\longmapsto f(\matr{U}\Gmat{i}{j}{\matr{\Utwo}}).
\end{split}
\end{equation}
When maximizing $\hij{i}{j}{\matr{U}}$, we can only consider \emph{rotations}, i.e.,  $\matr{\Psi}= \matr{\Utwo}(c,s_1,s_2) =$  
\begin{align}
\matr{\Utwo}(c,s_1,s_2)
&= \begin{bmatrix}
 c & -s \\
  s^{\ast} & c
 \end{bmatrix}
= \begin{bmatrix}
 c & -(s_1+\ui s_2) \\
  s_1-\ui s_2 & c
 \end{bmatrix}= \begin{bmatrix}
 \cos\theta & -\sin\theta \ue^{\ui\phi} \\
  \sin\theta \ue^{-\ui\phi} & \cos\theta
 \end{bmatrix}
\label{unitary-para-2}
\end{align}
where $c\in \RR^{+}$, $s= s_1 + \ui s_2\in\CC$, $c^2+|s|^2 =1$. 
This is due to the fact that  \eqref{cost-fn-general} and \eqref{cost-fn-general-trace}  are  invariant under
multiplications of columns of $\matr{U}$ by  scalars  from $\mathbb{T}$,  i.e., 
\begin{equation}\label{eq:invariance_scaling}
f(\matr{U}) = f(\matr{U}\matr{S}), \quad \text{ for all } \matr{S} =  
\begin{bmatrix}z_1 & &  0 \\  & \ddots &  \\ 0 & & z_n \end{bmatrix},\quad z_i\in\mathbb{T}, 1\leq i\leq n.
\end{equation}
As in the  matrix case \cite{GoluV96:jhu}, we  refer to  $\Gmat{i}{j}{\matr{\Utwo}}$ with $\matr{\Utwo}$ of the form \eqref{unitary-para-2} as \emph{Givens rotations}, and to maximizers of $\hij{i}{j}{\matr{U}}$ as \emph{Jacobi rotations}.
As shown in \Cref{sec:deriv_forms}, for any cost function \eqref{cost-fn-general} or \eqref{cost-fn-general-trace} with $d\le 3$, maximization of  $\hij{i}{j}{\matr{U}}$  is equivalent\footnote{This fact is known for  special cases \eqref{eq-cost-jade}-\eqref{cost-4-order}. For  $d > 3$, a closed form solution does not exist in general in the complex case, as shown in  \Cref{sec:deriv_forms}, where a form of $\hij{i}{j}{\matr{U}}$ is derived for any $d$.} to finding the leading eigenvalue/eigenvector pair of a $3\times 3$ symmetric matrix;  hence  updates are very cheap.
Therefore, in this paper, we mostly focus on the case $d\le3$. 

A typical choice of pairs $(i_k,j_k)$ (used  in \cite{cardoso1993blind,cardoso1996jacobi,de1997signal,comon2001source}) is, e.g.,  \emph{cyclic-by-row},  
\begin{equation}\label{equation-Jacobi-C}
  \begin{split}
&(1,2) \to (1,3) \to \cdots \to (1,n) \to  (2,3) \to \cdots \to (2,n) \to  \cdots  \to  (n-1,n)  \to \\
&(1,2) \to (1,3) \to \cdots
  \end{split}
  \end{equation}
The convergence of the iterations for cyclic Jacobi  algorithms is unknown, except in the single matrix case\footnote{or a similar case of a pair of commuting matrices \cite{BunsBM93:simax}. These cases are special because,  the matrices can be always diagonalized (the minimal value of the off-norm is zero).} \cite{GoluV96:jhu}. 
Most of the results for the matrix case are on the convergence of $f(\matr{U}_k)$ to $\|\matr{A}\|^2$ (or the off-norm \eqref{eq-off-norm} to zero).
The rate is linear and asymptotically quadratic, for the cyclic strategies of choice of pairs and a class of other strategies,  see  \cite[\S 8.4.3]{GoluV96:jhu} and \cite{HariB17:cyclic,HariB19:complex} for an overview.
Moreover, the  result  \cite{Masc95:simax} guarantees that  in this case $\matr{U}_k^{\H} \matr{A} \matr{U}_k$ converges to a  diagonal matrix.
However, this implies convergence  of $\matr{U}_k$ to a limit point only if the eigenvalues of $\matr{A}$ are distinct 
(for multiple eigenvalues, convergence of subspaces is proved \cite{Drma10:simax}).
All these results are specific to matrices, and cannot be directly applied to our case.
Finally,  note that an extension of the Jacobi algorithm to compact Lie groups was proposed in \cite{Klei04:lie}, but their setup is  different: it is the notion of diagonality of a matrix that is generalized to Lie groups in \cite{Klei04:lie}, while we consider  higher-order cost functions.

\subsection{Jacobi-G algorithm and an overview of results}
Recently, a gradient-based Jacobi algorithm (Jacobi-G) was proposed  \cite{IshtAV13:simax}  in a context of optimization on orthogonal group.
Its weak  convergence  was shown in \cite{IshtAV13:simax}  and global convergence for real matrix and 3rd order tensor case was proved in \cite{LUC2018}.
 In this subsection, we introduce a complex generalization of the  Jacobi-G algorithm (\Cref{Jacobi-G-general}).
The main idea behind the algorithm is to  choose  Givens transformations that are  well aligned with  the Riemannian gradient\footnote{The definition of Riemannian gradient is postponed to \Cref{sec:unitary_geom}.} of $f$ denoted by $\ProjGrad{f}{\cdot}$.

\begin{algorithm}
\caption{General Jacobi-G algorithm}\label{Jacobi-G-general}
{\bf Input:} A differentiable $f:\UN{n}\rightarrow\RR$,  constant $0<\delta\le\sqrt{2}/n$,  starting point $\matr{U}_{0}$.\\
{\bf Output:} Sequence of iterations $\matr{U}_{k}$.
\begin{itemize}
\item {\bf For} $k=1,2,\ldots$ until a stopping criterion is satisfied do
\item\quad Choose an index pair  $(i_k,j_k)$  satisfying
\begin{equation}\label{inequality-gra-based}
\|\ProjGrad{\hij{i_k}{j_k}{\matr{\matr{U}_{k-1}}}}{\matr{I}_2}\| \geq \delta
\|\ProjGrad{f}{\matr{U}_{k-1}}\|.
\end{equation}
\item\quad Find $\matr{\Utwo}_k$ that maximizes $h_k(\matr{\Utwo})\eqdef\hij{i_k}{j_k}{\matr{U}_{k-1}}(\matr{\Utwo})$.
\item\quad Update $\matr{U}_k = \matr{U}_{k-1} \Gmat{i_k}{j_k}{\matr{\Utwo}_k}$.
\item {\bf End for}
\end{itemize}
\end{algorithm}

It is shown in Section~\ref{sec:deriv_forms} that  it is always possible to find  $(i_k,j_k)$ satisfying \eqref{inequality-gra-based}, provided   $\delta\le\sqrt{2}/n$ (the meaning of $\delta$ will be also explained).
Next, we summarize main results on convergence of \Cref{Jacobi-G-general} for \eqref{cost-fn-general} and \eqref{cost-fn-general-trace}, $d\le 3$.
\begin{itemize}
\item \Cref{theorem-local-gene}: we show  that, similarly to the algorithm of \cite{IshtAV13:simax}, the weak convergence takes place ($\ProjGrad{f}{\matr{U}_k} \to 0$), which implies that every accumulation point $\overline{\matr{U}}$ of the sequence $\{\matr{U}_k\}$ is a stationary point; moreover, we are able to retrieve global convergence rates along the lines of \cite{boumal2016global}.

\item \Cref{thm:accumulation_points}: if an accumulation point $\overline{\matr{U}}$ satisfies  regularity conditions (i.e., restrictions $\hij{i}{j}{\overline{\matr{U}}}$, for all $i < j$, have semi-strict local maxima at $\matr{I}_2$), then $\overline{\matr{U}}$ is the only limit point  of $\{\matr{U}_k\}$; if in addition, the rank of the Hessian at $\overline{\matr{U}}$ is maximal (i.e., equal to $n(n-1)$), then the speed of convergence is linear.

\item \Cref{thm:local_linear}: if $\matr{U}_*$ is a semi-strict local maximum of $f$, then \Cref{Jacobi-G-general} converges linearly to  $\matr{U}_*$ (or an equivalent point) when started at any point in a neighborhood of $\matr{U}_*$.
\end{itemize}

We  eventually provide in \Cref{sec:examples}   examples of tensor and matrix diagonalization problems where the regularity conditions are satisfied. 
In the results listed above, we use the notion of semi-strict local maximum due to   invariance   of the cost function  with respect to \eqref{eq:invariance_scaling}.
This makes the Riemannian Hessian rank-deficient (rank at most $n(n-1)$) at any stationary point, hence the maxima cannot be strict.
We use the following tools to overcome this issue:
\begin{itemize}
\item Morse-Bott property that generalizes  Morse property at a stationary point;
\item quotient manifold ${\tUN{n}}$: factorizing  $\UN{n}$ by the equivalence relation in \eqref{eq:invariance_scaling}. 
\end{itemize}

\section{Unitary group as a real manifold}\label{sec:unitary_geom}

This section contains all necessary facts about the unitary group, derivatives of the cost functions, geodesics, etc. 

\subsection{Wirtinger calculus}
First, we introduce the following real-valued inner product\footnote{In some literature \cite{abrudan2008steepest}, a different inner product $\frac{1}{2} \Re \left({\trace{\matr{X}^{\H}\matr{Y}}}\right)$ is adopted. We prefer a definition that is compatible with the  Frobenius norm  $\RInner{\matr{X}}{\matr{X}} = \|\matr{X}\|^2$ } on $\CC^{m\times n}$. For  $\matr{X} = \matr{X}^{\R}+ \ui\matr{X}^{\I},\matr{Y} = \matr{Y}^{\R}+ \ui\matr{Y}^{\I}\in \CC^{m\times n}$, we denote 
\begin{equation}\label{eq:RInner}
\RInner{\matr{X}}{\matr{Y}} \eqdef 
\langle\matr{X}^{\R},\matr{Y}^{\R} \rangle + \langle\matr{X}^{\I},\matr{Y}^{\I} \rangle  =
\Re \left({\trace{\matr{X}^{\H}\matr{Y}}}\right).
\end{equation}
This makes $\CC^{m\times n}$ a real Euclidean space  of dimension $2mn$. 

Note that a function  $f:\CC^{m\times n}\rightarrow\RR$ is never holomorphic unless it is  constant; therefore we do not require $f$ to be complex differentiable, but  differentiable with respect to the real and imaginary parts. 
We use a shorthand notation $\deriv{f}{\matr{X}^{\R}},\deriv{f}{\matr{X}^{\I}} \in \RR^{m\times n}$ for  matrix derivatives with respect to  real and imaginary parts of $\matr{X} \in \CC^{m\times n}$.
The Wirtinger derivatives   $\deriv{f}{\matr{X}},\deriv{f}{\matr{X}^*} \in \CC^{m\times n}$ are standardly defined  \cite{abrudan2008steepest,brandwood1983complex,krantz2001function} as
\begin{equation*}
\deriv{f}{\matr{X}} := \frac{1}{2}\left(\deriv{f}{\matr{X}^{\R}}- \ui\deriv{f}{\matr{X}^{\I}}\right), \quad
\deriv{f}{\matr{X}^*} := \frac{1}{2}\left(\deriv{f}{\matr{X}^{\R}}+ \ui\deriv{f}{\matr{X}^{\I}}\right).
\end{equation*}
The matrix Euclidean gradient of $f$ with respect to the inner product \eqref{eq:RInner} becomes
\begin{equation*}
\nabla^{(\R)} f(\matr{X}) = \deriv{f}{\matr{X}^{\R}}+ \ui\deriv{f}{\matr{X}^{\I}} = 2 \deriv{f}{\matr{X}^*}(\matr{X}).
\end{equation*}

\subsection{Riemannian gradient}
$\UN{n}$ can be viewed as an embedded real submanifold of $\CC^{n\times n}$ (see also \cite[Appendix C.2.6]{Hall15:lie}).  By \cite[\S 3.5.7]{Absil08:Optimization}, the tangent space to $\UN{n}$ is associated with an $n^2$-dimensional $\RR$-linear subspace  of $\CC^{n\times n}$:
\begin{equation*}
\mathbf{T}_{\matr{U}}\UN{n} = \{ \matr{X}\in\CC^{n\times n} :  \matr{X}^{\H} \matr{U} +  \matr{U}^{\H} \matr{X} = 0  \} =  \{ \matr{X}\in\CC^{n\times n} :  \matr{X} = \matr{U} \matr{Z}, \quad  \matr{Z} + \matr{Z}^{\H} = 0   \}.
\end{equation*} 
Recall that $\UN{n}$ is a matrix Lie group, with the Lie algebra of skew-Hermitian matrices $\mathfrak{u}(n) =  \{ \matr{Z}\in\CC^{n\times n} :  \matr{Z} + \matr{Z}^{\H} = 0   \}$ (which coincides with $\mathbf{T}_{\matr{I}_n}\UN{n}$ in our notation).
Then for $f: \CC^{n\times n} \to \RR$ differentiable in a neighborhood of $\UN{n}$,  the Riemannian gradient is just  the orthogonal projection of   $\nabla^{(\R)} f(\matr{U})$  on $\mathbf{T}_{\matr{U}}\UN{n}$: 
\begin{align}
&\ProjGrad{f}{\matr{U}}   = \matr{U} \matr{\Lambda}(\matr{U}) \in \mathbf{T}_{\matr{U}}\UN{n},\quad\text{where}\label{eq-riem-grad} \\
&\matr{\Lambda}(\matr{U})  =\frac{\matr{U}^{\H}
\nabla^{(\R)} f(\matr{U}) - (\nabla^{(\R)} f(\matr{U}))^{\H}\matr{U}}{2}  =\matr{U}^{\H}
\deriv{f}{\matr{U}^{*}}(\matr{U}) - (\deriv{f}{\matr{U}^{*}}(\matr{U}))^{\H}\matr{U}.\label{eq:Lambda}
\end{align}
Note that $\matr{\Lambda}(\matr{U})$ is a skew-Hermitian matrix,  i.e.,  
\begin{equation}\label{eq:Lambda-antihermitian}
\Lambda_{ij}(\matr{U})=-(\Lambda_{ji}(\matr{U}))^{*},\ \ \ 1\leq i,j\leq n.
\end{equation}
In what follows, we will  use the exponential map  \cite[p.102]{Absil08:Optimization} $\text{Exp}_{\matr{U}}: \mathbf{T}_{\matr{U}}\UN{n}  \to \UN{n}$,
which maps 1-dimensional lines in the tangent space to geodesics and is given by 
\begin{equation}\label{eq:exp_map}
\text{Exp}_{\matr{U}}(\matr{U}\Omega) = \matr{U} \exp(\Omega),
\end{equation}
where $\exp(\cdot)$ is the matrix exponential.
We will frequently use the following relation between $\text{Exp}_{\matr{U}}$ and the Riemannian gradient.
For any $\Delta \in \mathbf{T}_{\matr{U}}\UN{n}$, we have
\begin{equation}\label{eq:gradient_exp_map}
\RInner{\matr{\Delta}}{\ProjGrad{f}{\matr{U}}}  =
\left.\left( \frac{d}{dt} f (\text{Exp}_{\matr{U}}(t\matr{\Delta}) )\right)\right|_{t=0}. 
\end{equation}
We also need the following fact about the case of scale-invariant functions.
\begin{lemma}\label{lem:grad_scale_invariant}
 Assume that $f: \UN{n} \to \RR$ satisfies  the invariance property \eqref{eq:invariance_scaling}.
 Then for any $\matr{U} \in \UN{n}$ and  $\matr{S}$ as in \eqref{eq:invariance_scaling} it holds that
 \[
 \ProjGrad{f}{\matr{U}\matr{S}} =  \ProjGrad{f}{\matr{U}}{\matr{S}}.
 \]
\end{lemma}
\begin{proof}
By  the chain rule, we have  $\deriv{f}{\matr{U}^*}(\matr{U}\matr{S}) = \left(\deriv{f}{\matr{U}^*}(\matr{U})\right)\matr{S}$. Therefore,
\[
\ProjGrad{f}{\matr{U}\matr{S}}  = 
\matr{U}\matr{S}
\left(\matr{S}^{\H}\matr{U}^{\H} \deriv{f}{\matr{U}^{*}}(\matr{U})\matr{S} - (\deriv{f}{\matr{U}^{*}}(\matr{U})\matr{S})^{\H}\matr{U}\matr{S}\right) = \matr{U} \matr{\Lambda}(\matr{U}) \matr{S},
\] 
where the last equality follows from  $\matr{S}\matr{S}^{\H} = \matr{I}$.
\end{proof}

\subsection{Derivatives for elementary rotations}
This section contains general facts about derivatives of $\hij{i}{j}{\matr{U}}$. 
First, for  $i \neq j$ we introduce a useful projection operator $\mathcal{P}_{i,j}: \CC^{n\times n} \to\CC^{2\times 2}$ that extracts a  submatrix of  $\matr{X} \in \CC^{n\times n}$ as follows:
\begin{equation}\label{eq:restriction}
\mathcal{P}_{i,j} (\matr{X}) = \begin{bmatrix}
{X}_{ii} & {X}_{ij} \\
{X}_{ji} & {X}_{jj}
 \end{bmatrix}.
 \end{equation}
 Its adjoint operator is $\mathcal{P}_{i,j}^{\T} : \CC^{2\times 2} \to \CC^{n\times n}$,  i.e.,  
 \begin{equation}\label{eq:restriction_adjoint}
 \mathcal{P}_{i,j}^{\T} \left( \begin{bmatrix} a & c \\ b & d \end{bmatrix}\right) =
  { \bbordermatrix{
  & & i& & j & \cr
   &\matr{0} & \vdots & & \vdots & \matr{0}&\cr
  i &\cdots &  a & &  c & \cr
   & & &  & &\cr
 j  &\cdots &  b & &  d & \cr
   & \matr{0} & & & & \matr{0} }}.
 \end{equation}
Note that   for the Givens transformation in \eqref{eq:givens_trans} we have 
\begin{equation*}
\mathcal{P}_{i,j} (\Gmat{i}{j}{\matr{\Utwo}}) = \matr{\Utwo},
\end{equation*}
which makes it easy to express  the Riemannian gradient of  $\hij{i}{j}{\matr{U}}$  through that of $f$.
\begin{lemma}\label{lem:ProjGradSubmatrix}
The Riemannian gradient of  $\hij{i}{j}{\matr{U}}$ defined in \eqref{eq-func-h-new}  at the identity matrix $\matr{I}_2$ is a submatrix of the matrix $\matr{\Lambda}(\matr{U})$ defined in \eqref{eq:Lambda}:
\begin{align}
\ProjGrad{\hij{i}{j}{\matr{U}}}{\matr{I}_2} = \mathcal{P}_{i,j} (\matr{\Lambda}(\matr{U}))
= \begin{bmatrix}
 \Lambda_{ii}(\matr{U}) &
 \Lambda_{ij}(\matr{U}) \\
 \Lambda_{ji}(\matr{U})  &  \Lambda_{jj}(\matr{U}) 
 \end{bmatrix}.\label{eq-gradient-h}
\end{align}
\end{lemma}
\begin{proof}
Denote $h = \hij{i}{j}{\matr{U}}$ for simplicity.
For any $\matr{\Delta} \in \Tang{\matr{I}_2}{\UN{2}}$, by \eqref{eq:gradient_exp_map}
\begin{equation*}
\begin{split}
&\RInner{\matr{\Delta}}{\ProjGrad{h}{\matr{I}_2}}  =
\left.\left( \frac{d}{dt} h (\text{Exp}_{\matr{I}_2}(t\matr{\Delta}) )\right)\right|_{t=0} = 
\left.\left( \frac{d}{dt} f (\matr{U} \Gmat{i}{j}{\text{Exp}_{\matr{I}_2}(\matr{\Delta} t)} )\right)\right|_{t=0} \\
&= \left.\left( \frac{d}{dt} f (\text{Exp}_{\matr{U}} (\matr{U} \mathcal{P}_{i,j}^{\T}(\matr{\Delta} )t )) \right)\right|_{t=0} 
 = \RInner{\matr{U} \mathcal{P}_{i,j}^{\T}(\matr{\Delta})}{\ProjGrad{f}{\matr{U}}} =
\RInner{\matr{\Delta}}{\mathcal{P}_{i,j} (\matr{\Lambda}(\matr{U}))},
\end{split}
\end{equation*}
which completes the proof.
\end{proof}

\subsection{Quotient manifold}\label{sec:quotient}

In order to handle scale invariance, it is often convenient to work on the quotient manifold. We define  the action of $\mathbb{T}^{n}$ on $\UN{n}$ as 
\begin{equation*}
\matr{U} \cdot (z_1,\ldots, z_n) = \matr{U} \begin{bmatrix}z_1 & &  0 \\  & \ddots &  \\ 0 & & z_n \end{bmatrix}.
\end{equation*} 
Since the action of $\mathbb{T}^{n}$ on $\UN{n}$ is free and proper, the quotient manifold ${\tUN{n}} = \UN{n}/\mathbb{T}^{n}$ is well-defined.
In order to define the gradient and Hessians on ${\tUN{n}}$, we use the standard splitting into vertical and horizontal space
\begin{equation*}
\Tang{\matr{U}}{\UN{n}}  = \TangV{\matr{U}}{\UN{n}} \oplus \TangH{\matr{U}}{\UN{n}}, 
\end{equation*}
where $\TangH{\matr{U}}{\UN{n}}$ contains the skew-symmetric matrices with zero diagonal:
\begin{equation*}
\TangH{\matr{U}}{\UN{n}} = \{ \matr{X}\in\CC^{n\times n} :  \matr{X} = \matr{U} \matr{Z}, \quad  \matr{Z} + \matr{Z}^{\H} = 0, \quad \diag{\matr{Z}} = \vect{0}   \}.
\end{equation*}
An element $\widetilde{\matr{U}} \in {\tUN{n}}$  is then  represented by $\matr{U}$ and the tangent space $\Tang{\widetilde{\matr{U}}} {{\tUN{n}}}$ is identified with $\TangH{{\matr{U}}}{{\UN{n}}}$, see \cite[\S 3.5.8]{Absil08:Optimization}.
Moreover, the Riemannian metric on ${\tUN{n}}$  is  defined as
\begin{equation*}
\langle\widetilde{\xi}, \widetilde{\eta}\rangle_{\Tang{\widetilde{\matr{U}}} {{\tUN{n}}}} = \langle\xi, \eta\rangle_{\Tang{{\matr{U}}} {{\UN{n}}}},
\end{equation*}
because the inner product  is invariant with respect to the choice of representative $\matr{U}$,  see  \cite[Section 3.6.2]{Absil08:Optimization}. 
This makes ${\tUN{n}}$ a Riemannian manifold;
the natural projection $\pi: \matr{U} \mapsto \widetilde{\matr{U}}$ then becomes a Riemannian submersion.

Due to the invariance property \eqref{eq:invariance_scaling}, the function $f$ is, in fact, defined on ${\tUN{n}}$ (we will denote the corresponding function  by  $\widetilde{f}: {\tUN{n}} \to\RR$).
\begin{remark}\label{lem-zero-diagonal}
 As shown in \cite[eqn. (3.39)]{Absil08:Optimization}, for any  $f$ satisfying the scale invariance property, we have    \eqref{eq:invariance_scaling}, $\ProjGrad{f}{\matr{U}} \in  \TangH{\matr{U}}{\UN{n}}$ (which naturally represents the gradient of $\widetilde{f}$ in ${\tUN{n}}$).
Therefore, in particular, the main diagonal of $\matr{\Lambda}({\matr{U}})$ is zero.
\end{remark}

\begin{remark}\label{exp-map-quotient}
Note that as in \cite[Thm. A.15]{MassA20:quotient}, for any $\matr{Z} \in \TangH{\matr{U}}{\UN{n}}$  the geodesic 
\begin{equation}\label{eq-exp-map-quotient}
\gamma(t) = \text{Exp}_{\matr{U}}(\matr{Z} t)
\end{equation}
is horizontal  (i.e. its derivative stays in the horizontal space $\dot{\gamma}(t) \in \TangH{\gamma(t)}{\UN{n}}$).
Thus the exponential map in the quotient manifold ${\tUN{n}}$ is also defined by \eqref{eq-exp-map-quotient}.
\end{remark}
Finally, we make  remarks about the two-dimensional manifold of $2\times 2$ rotations ${\tUN{2}}$.
\begin{remark}
The matrices $\matr{\Utwo}(c,s_1,s_2)$ defined in \eqref{unitary-para-2}, in fact, parametrize  ${\tUN{2}}$.
\end{remark}
\begin{remark}
Since all the elements on the diagonals are zero, the tangent space $\Tang{\widetilde{\matr{U}}} {{\tUN{n}}}$ to the $n(n-1)$-dimensional manifold ${\tUN{n}}$ can be decomposed as a direct sum of $\frac{n(n-1)}{2}$ copies of $\Tang{\widetilde{\matr{I}_2}} {{\tUN{2}}}$ (spaces of $2\times2$ skew-symmetric matrices with zero diagonal corresponding to different pairs $(i,j)$);
this can be also seen from \Cref{lem:ProjGradSubmatrix}.
\end{remark}

\subsection{Riemannian Hessian and stationary points}
For a Riemannian manifold $\mathcal{M}$ and a $C^2$ function $f: \mathcal{M} \to \RR$, the  Riemannian Hessian at $x\in\mathcal{M}$ is either defined as a linear map $\mathbf{T}_{x}\mathcal{M} \to \mathbf{T}_{x}\mathcal{M}$ or as a bilinear form on $\mathbf{T}_{x}\mathcal{M}$; the usual definition is based on the Riemannian connection \cite[p.105]{Absil08:Optimization}. 

For our purposes, for simplicity, we assume that the exponential map $\text{Exp}_{\matr{x}}: \mathbf{T}_{x}\mathcal{M} \to \mathcal{M}$ is given, and adopt the following definition based on \cite[Proposition 5.5.4]{Absil08:Optimization}.
The Riemannian Hessian $\Hess{f}{x}$ is the linear map $\mathbf{T}_{x}\mathcal{M} \to \mathbf{T}_{x}\mathcal{M}$  defined by
\[
\Hess{f}{x} = \Hess{(f \circ\text{Exp}_{\matr{x}})}{\matr{0}_x},
\]
where $\matr{0}_x$ is the origin in the tangent space, and $\Hess{g}{\matr{0}_x}$ is the Euclidean Hessian of $g: \mathbf{T}_{x}\mathcal{M}  \to \RR$.
Hence, similarly to \eqref{eq:gradient_exp_map}, there is the following expression for the values of Riemannian Hessian as a quadratic form at $\matr{\Delta} \in \mathbf{T}_{x}\mathcal{M}$:
\begin{equation}\label{eq:hessian_quad_form}
\RInner{\HessAppl{f}{x}{\matr{\Delta}}}{\matr{\Delta}} = 
\left.\left( \frac{d^2}{dt^2} f (\text{Exp}_{x}(t\matr{\Delta}) )\right)\right|_{t=0}. 
\end{equation}
The Riemannian  Hessian gives  necessary  and sufficient conditions of local extrema (see, for example, \cite[Theorem 4.1]{rapcsak1991}).
\begin{itemize}
\item If $x$ is a local maximum of $f$ on $\mathcal{M}$, then $\Hess{f}{x} \preceq 0$ (negative semidefinite);
\item If $\ProjGrad{f}{x} = 0$ and $\Hess{f}{x}  \prec_{\mathbf{T}_{x}\mathcal{M}} 0$  (i.e., $\Hess{f}{x}  \preceq 0$  and $\rank{\Hess{f}{x}} = \dim (\mathcal{M})$), then $f$ has a strict local maximum at $x$.
\end{itemize}
Finally, we distinguish stationary points with nonsingular Riemannian Hessian.
\begin{definition}\label{def:nondegenerate}
A stationary point ($x \in \mathcal{M}$, $\ProjGrad{f}{x} = 0$) is called \emph{non-degenerate} if $\Hess{f}{x}$ is nonsingular on $\mathbf{T}_{x}\mathcal{M}$.
\end{definition}
In our case, a stationary point is never non-degenerate, as shown below.  
\begin{lemma}\label{lem:f_invariant_Hessian}
 Assume that $f: \UN{n} \to \RR$ satisfies  the invariance property \eqref{eq:invariance_scaling}. 
Let $\matr{U}$ be a stationary point ($\ProjGrad{f}{\matr{U}} = 0$) and
\begin{equation}\label{eq:Z_k}
\matr{Z}_k = \begin{bmatrix} \vect{0} &\cdots & \vect{0} & \ui \vect{u}_k &\vect{0} &\cdots &\vect{0} \end{bmatrix}=  \matr{U} \matr{\Omega}_k \in \mathbf{T}_{\matr{U}}   \UN{n},
\end{equation}
where $\matr{\Omega}_k = i \vect{e}_k \vect{e}_k^{\T}$ (where $\vect{e}_k$ is the k-th unit vector).
Then  $\HessAppl{f}{\matr{U}}{\matr{Z}_k} = \vect{0}$ (i.e., all  $\matr{Z}_k$ are in the kernel of $\Hess{f}{\matr{U}}$). 
In particular, 
 $\rank{\Hess{f}{\matr{U}}} \le n(n-1)$. 
\end{lemma}
\begin{proof}
Let $\gamma :\RR\to \UN{n}$ be a curve defined as
$\gamma(t) = \text{Exp}_{\matr{U}}(t\matr{Z}_k)$,
with $\gamma(0)=\matr{U}$, $\gamma'(0)=\matr{Z}_k$.
By \cite[Def. 5.5.1]{Absil08:Optimization}, \cite[(5.15)]{Absil08:Optimization} and  \Cref{lem:grad_scale_invariant}, we obtain
\begin{align*}
\HessAppl{f}{\matr{U}}{\matr{Z}_k}&= {\bf P}_{\matr{U}}\left(\left.\frac{d}{d t}\ProjGrad{f}{\gamma(t)}\right|_{t=0}\right) = {\bf P}_{\matr{U}}\left(\left.\frac{d}{d t}\ProjGrad{f}{\matr{U}}\exp(t \matr{\Omega}_k) \right|_{t=0}\right)  \\
&= {\bf P}_{\matr{U}}\left(\ProjGrad{f}{\matr{U}}  \matr{\Omega}_k\right) =\frac{\matr{U}}{2} \left(\matr{\Lambda}(\matr{U}) \matr{\Omega}_k  -\matr{\Omega}_k \matr{\Lambda}(\matr{U})  \right).
\end{align*} 
Note that $\matr{U}$ is a stationary point. Then  $\matr{\Lambda}(\matr{U}) = 0$, and  thus  $\HessAppl{f}{\matr{U}}{\matr{Z}_k} = \matr{0}$. 
Since  $\{\matr{Z}_k\}_{k=1}^n$,  are linearly independent,  $\rank{\Hess{f}{\matr{U}}} \le n(n-1)$. 
\end{proof}

\section{Finding Jacobi rotations and derivatives for complex forms}\label{sec:deriv_forms}

\subsection{On correctness of  Jacobi-G}
The following fact follows from  \cref{lem:ProjGradSubmatrix}. 
\begin{corollary}\label{cor:ProjGradSubmatrixInequality}
Let $f$ and $\hij{i}{j}{\matr{\matr{U}}}$ be as in \Cref{lem:ProjGradSubmatrix}. Then
\begin{equation*}
\max_{1\le i<j \le n}
\|\ProjGrad{\hij{i}{j}{\matr{\matr{U}}}}{\matr{I}_2}\| \geq \frac{\sqrt{2}}{n}
\|\ProjGrad{f}{\matr{U}}\|.
\end{equation*}
\end{corollary}
\begin{proof}
 By \eqref{eq-riem-grad} and \Cref{lem:ProjGradSubmatrix}, we see that 
\begin{equation*}
\begin{split}
\|\ProjGrad{f}{\matr{U}}\|^2 &= \|\matr{\Lambda}(\matr{U})\|^2  = \sum\limits_{i,j=1}^{n,n} |\Lambda(\matr{U})_{i,j}|^2 \le \frac{n^2}{2} \max_{1\le i<j \le n} \|\ProjGrad{\hij{i}{j}{\matr{\matr{U}}}}{\uI_2}\|^2. \qed
\end{split}
\end{equation*}
\end{proof}

\begin{remark}
\Cref{cor:ProjGradSubmatrixInequality}  implies that for any differentiable $f$ it is always possible to find $(i_k,j_k)$ satisfying the inequality \eqref{inequality-gra-based}, provided $\delta \le \sqrt{2}/n$.
\end{remark}
In fact, it gives an explicit way to find such a pair, as shown by the following remark.

\begin{remark}\label{rem:choosing_pair}
From \Cref{lem:ProjGradSubmatrix,lem-zero-diagonal}, the condition \eqref{inequality-gra-based} becomes
\begin{equation}\label{inequality-gra-Lambda}
\sqrt{2}|\Lambda_{i_k,j_k}| \ge \delta \|\matr{\Lambda}\|,
\end{equation}
where $\Lambda = \Lambda(\matr{U}_{k-1}) = \matr{U}^{\H}_{k-1}\ProjGrad{f}{\matr{U}_{k-1}}$ is as in \eqref{eq:Lambda}.
Thus the pair can be selected by looking at the elements of $\matr{\Lambda}$, for example, according to one of the strategies:
(a) choose the maximal modulus element of $\matr{\Lambda}$; or
(b) choose the first pair (e.g., in cyclic order) that satisfies \eqref{inequality-gra-Lambda};  if $\delta$ is small, then  \eqref{inequality-gra-Lambda} is most of the time satisfied.
\end{remark}

\subsection{Elementary rotations}
First of all,  our  cost functions   that  satisfy the invariance property \eqref{eq:invariance_scaling}; hence 
 the restriction \eqref{eq-func-h-new}  is also scale-invariant 
\begin{equation}\label{eq:hij_invariance}
\hij{i}{j}{\matr{U}}(\matr{\Utwo}) = \hij{i}{j}{\matr{U}}\left(\matr{\Utwo} \left[\begin{smallmatrix}z_1 & 0 \\ 0 & z_2\end{smallmatrix}\right]\right),\quad\text{for all } z_1, z_2\in \mathbb{T}. 
\end{equation}
Hence, we can  restrict to  matrices $\matr{\Utwo} = \matr{\Utwo}(c,s_1,s_2)$ defined in
\eqref{unitary-para-2},
and maximize
\[
\hij{i}{j}{\matr{U}}(c,s_1,s_2)  \eqdef \hij{i}{j}{\matr{U}}\left(\matr{\Utwo}(c,s_1,s_2)\right) = \hij{i}{j}{\matr{U}}\left( \left[\begin{smallmatrix}
 c & -(s_1+is_2) \\
 s_1-is_2 & c
 \end{smallmatrix}\right]\right);
\]
next, we show how to maximize $\hij{i}{j}{\matr{U}}(c,s_1,s_2)$ for  cost functions \eqref{cost-fn-general} and \eqref{cost-fn-general-trace}.
\begin{proposition}\label{lem-cost-quadratic-form}
For a cost function $f$ of the form  \eqref{cost-fn-general} and \eqref{cost-fn-general-trace} with  $d \le 3$, its restriction for any pair $(i,j)$ and   $\matr{U} \in \UN{n}$  can be expressed as a quadratic form 
\begin{align}
&\hij{i}{j}{\matr{U}}(c,s_1,s_2) = \vect{r}^{\T} \Gamij{i}{j}{\matr{U}}\vect{r} + C,\quad \text{where}\label{eq-cost-quadratic-form} \\
&\vect{r} = \vect{r}(c,s_1,s_2)  \eqdef\left[\begin{smallmatrix}2c^2-1&\; -2cs_1&\; -2cs_2\end{smallmatrix}\right]^{\T} = \left[\begin{smallmatrix}\cos 2\theta&\; -\sin 2\theta  \cos\phi &\; -\sin 2\theta  \sin\phi \end{smallmatrix}\right]^{\T}, \label{eq:w_definition}
\end{align}
$\Gamij{i}{j}{\matr{U}}\in \RR^{3\times 3}$ is a symmetric matrix and $C$ is a constant, whose entries depend polynomially on the real and imaginary parts of $\matr{U}$ and of tensors $\tens{B}$ or $\tens{A}^{(\ell)}$.
\end{proposition}
In fact, \Cref{lem-cost-quadratic-form}  was already known for special cases of problems \eqref{eq-cost-jade}--\eqref{cost-4-order} (see \cite[Ch. 5]{comon2010handbook} for an overview); in its general form, \Cref{lem-cost-quadratic-form} is a special case of a general result (\Cref{thm:elementary_updates} in  \cref{sec:elem_updates_general}) that establishes the form of $\hij{i}{j}{\matr{U}}$ for any  order $d$.

To illustrate the idea, we give an example for joint matrix diagonalization.
\begin{example}\label{example-quadratic-form-jade}
For the  function \eqref{eq-cost-jade}, denote $\matr{W}^{(\ell)}=\matr{U}^{\H}\matr{A}^{(\ell)}\matr{U}$, so that  $f(\matr{U}) = \sum\limits_{\ell=1}^{L} \|\diag{\matr{W}^{(\ell)}}\|^2$.
Then it is known \cite{cardoso1996jacobi} that\footnote{Note that these expressions can be simplified for Hermitian matrices $\matr{W}^{(\ell)}$, because in this case
$W_{ij}+W_{ji} =  2{W}^{\R}_{ij}$ and $-\ui (W_{ij} - W_{ji}) = 2 W^{\I}_{ij}$.} 
\begin{align*}
\Gamij{i}{j}{\matr{U}} & =\frac{1}{2}\sum\limits_{\ell=1}^{L} \left( |{{W}^{(\ell)}_{jj}}+{{W}^{(\ell)}_{ii}}|^2  \matr{I}_3 +  \R\left(\vect{z}(\matr{W}^{(\ell)}) \matr{z}^{\H}(\matr{W}^{(\ell)})\right)\right),\quad\mbox{where} \\
\vect{z}(\matr{W}) & \eqdef \begin{bmatrix} W_{jj}-W_{ii} & W_{ij}+W_{ji} & -\ui (W_{ij}-W_{ji}) \end{bmatrix}^{\T}.
\end{align*}
\end{example}
Similar expressions exist for the cost functions  \eqref{eq-cost-3-tensor} (see \cite[(9.29)]{de1997signal} and \cite[Section 5.3.2]{comon2010handbook})  and  \eqref{cost-4-order} (see  \cite{comon2001source}),  but we omit them due to space limitations, and also because \Cref{lem-cost-quadratic-form} supersedes all these results. 

\begin{remark}\label{remark-eigenvectors}
By \Cref{lem-cost-quadratic-form}, the maximization of $\hij{i}{j}{\matr{U}}(c,s_1,s_2)$ is equivalent to maximization of  $\vect{r}^{\T}\Gamij{i}{j}{\matr{U}} \vect{r}$ subject to $\|\vect{r}\| =1$. 
 Thus a maximizer of $\hij{i}{j}{\matr{U}}(c,s_1,s_2)$ can be obtained from an eigenvector of $\Gamij{i}{j}{\matr{U}}$, which we summarize in \Cref{alg-eigenvector-of-Gamma}.
Note that we can choose the maximizer such that $c \ge \frac{\sqrt{2}}{2}$.
\end{remark}

\begin{algorithm}
\caption{Finding Jacobi rotations}\label{alg-eigenvector-of-Gamma} 
{\bf Input:} Point $\matr{U}$, pair $(i,j)$.\\
{\bf Output:} A maximizer $(c,s_1,s_2)$ of $\hij{i}{j}{\matr{U}}(c,s_1,s_2)$.
\begin{itemize}
\item Build $\matr{\Gamma} = \Gamij{i}{j}{\matr{U}}$ according to \Cref{lem-cost-quadratic-form}.
\item Find a leading eigenvector $\vect{w}$ corresponding to the maximal eigenvalue of $\matr{\Gamma}$ (with normalization $\|\vect{w}\| = 1$, $w_1 \ge 0$).
\item Choose $\theta \in \left[0,\frac{\pi}{4}\right]$  and $s_1,s_2$ by  setting 
$\theta = \frac{\arccos(w_1)}{2} \in \left[0,\frac{\pi}{4}\right]$, 
$c = \cos (\theta) = \sqrt{\frac{w_1+1}{2}}\ge \frac{\sqrt{2}}{2}$, 
$s_1 = -\frac{w_2}{2 c}$,  $s_2 = -\frac{w_3}{2 c}$.
\end{itemize}
\end{algorithm}

\subsection{Riemannian derivatives for the  cost functions}
In this subsection, we link the Riemannian derivatives of $\hij{i}{j}{\matr{U}}$  with  the entries of the matrix $\Gamij{i}{j}{\matr{U}}$.
\begin{lemma}\label{prop:dir_deriv}
Let $\hij{i}{j}{\matr{U}}$ satisfy  \eqref{eq:hij_invariance} and be expressed as in \eqref{eq-cost-quadratic-form}.
Then 
\begin{align*}
\ProjGrad{\hij{i}{j}{\matr{U}}}{\uI_2} = 2
\begin{bmatrix}
  0 & \Gamij{i}{j}{\matr{U}}_{12} + \ui \Gamij{i}{j}{\matr{U}}_{13}\\
-\Gamij{i}{j}{\matr{U}}_{12} + \ui \Gamij{i}{j}{\matr{U}}_{13}& 0
 \end{bmatrix},
\end{align*}
hence, in particular, ${\Lambda}_{ij}(\matr{U}) = 2(\Gamij{i}{j}{\matr{U}}_{12} + \ui \Gamij{i}{j}{\matr{U}}_{13})$ by \Cref{lem:ProjGradSubmatrix}.
\end{lemma}
\begin{proof}
 Denote $h=\hij{i}{j}{\matr{U}}$ and $\matr{\Gamma}=\Gamij{i}{j}{\matr{U}}$.
By \eqref{eq:Lambda-antihermitian} and \Cref{lem-zero-diagonal},  we see that $\ProjGrad{h}{\uI_2}$ is skew-Hermitian, and is decomposable as 
 $\ProjGrad{h}{\uI_2}   = {2\omega_1} \matr{\Delta}_1 + {2\omega_2} \matr{\Delta}_2$, 
\begin{align}
& \text{where}\quad\matr{\Delta}_1 = 
\begin{bmatrix}
 0 &  -\frac{1}{2} \\
\frac{1}{2}  & 0
 \end{bmatrix}, \quad \matr{\Delta}_2 = 
\begin{bmatrix}
0  &  -\frac{\ui}{2} \\
-\frac{\ui}{2}  & 0
 \end{bmatrix}.\label{eq:T_U2_basis}
\end{align}
Note that $\{\matr{\Delta}_1,\matr{\Delta}_2\}$ is an orthogonal basis of  $\Tang{\matr{I}_2}{{\tUN{2}}}$.
Since  $\|\matr{\Delta}_1\|^2 = \|\matr{\Delta}_2\|^2 = 1/2$  
\begin{equation*}
\omega_k = \RInner{\matr{\Delta}_k}{\ProjGrad{h}{\matr{I}_2}} = \left.\left( \frac{d}{dt} h (e^{t\matr{\Delta}_k})\right)\right|_{t=0} 
\end{equation*}
 for  $k=1,\,2$. 
On the other hand, we have  
\begin{equation*}
\begin{split}
 h (e^{t\matr{\Delta}_1}) & = h \left( \begin{bmatrix}
 \cos \frac{t}{2} & -\sin \frac{t}{2} \\
\sin  \frac{t}{2} & \cos  \frac{t}{2}
 \end{bmatrix} \right) = h \left(\cos \frac{t}{2}, \sin \frac{t}{2}, 0\right) =   \overline{h} (\cos t, -\sin t, 0),\\
 h (e^{t\matr{\Delta}_2}) & = h \left(\begin{bmatrix}
 \cos \frac{t}{2} & -i\sin \frac{t}{2} \\
-i\sin \frac{t}{2} & \cos \frac{t}{2}
 \end{bmatrix} \right) =h \left(\cos \frac{t}{2}, 0, \sin \frac{t}{2}\right) =  \overline{h} (\cos t, 0, -\sin t),
 \end{split}
\end{equation*}
where  $\overline{h} (\vect{r}) =\vect{r}^{\T} \matr{\Gamma}  \vect{r}$.  
Since   $\nabla \overline{h} (\vect{r})  = 2 \matr{\Gamma}  \vect{r}$,  we have
\begin{equation*}
\omega_1 = -\frac{\partial \overline{h}}{\partial  r_2} (1,0,0) =  -2\Gamma_{21}, \quad \omega_2 = -\frac{\partial \overline{h}}{\partial  r_3} (1,0,0) =  -2\Gamma_{31},
\end{equation*}
which completes the proof.
\end{proof}
\begin{lemma}\label{prop:dir_deriv_Hessian}
For $\hij{i}{j}{\matr{U}}$ as in \Cref{prop:dir_deriv}, and the basis of $\Tang{\matr{I}_2}{{\tUN{2}}}$  as in \eqref{eq:T_U2_basis},
the Riemannian Hessian of $\widetilde{h}$ ($\hij{i}{j}{\matr{U}}$ on ${\tUN{2}}$) is 
\begin{equation}\label{eq:hess_2x2}
\Hess{\widetilde{h}}{\matr{I}_2}  = \HessPS{i}{j}{\matr{U}} \eqdef  2\left(\begin{bmatrix}\Gamij{i}{j}{\matr{U}}_{2,2} & \Gamij{i}{j}{\matr{U}}_{2,3} \\ \Gamij{i}{j}{\matr{U}}_{3,2} & \Gamij{i}{j}{\matr{U}}_{3,3}  \end{bmatrix} - \Gamij{i}{j}{\matr{U}}_{1,1}\matr{I}_2\right).
\end{equation}
\end{lemma}
\begin{proof}
We denote $h= \hij{i}{j}{\matr{U}}$ and $\matr{\Gamma}=\Gamij{i}{j}{\matr{U}}$  for simplicity, and take 
\begin{equation}\label{eq:omega_geodesic}
\matr{\Omega} = \alpha_1\matr{\Delta}_1 + \alpha_2\matr{\Delta}_2, \quad \text{where } \alpha_1, \alpha_2 \in \RR, \,\alpha_1^2 + \alpha_2^2 =1,
\end{equation}
and $\{\matr{\Delta}_1,\matr{\Delta}_2\}$ are as in \eqref{eq:T_U2_basis}.
Then $h(e^{t\matr{\Omega}})$  and its derivative can be expressed as 
\begin{align}
h(e^{t\matr{\Omega}}) &= h \left(\left[ \begin{smallmatrix}
 \cos \frac{t}{2} & -(\alpha_1 + \ui \alpha_2)\sin \frac{t}{2}  \\
(\alpha_1 -\ui\alpha_2) \sin  \frac{t}{2} & \cos  \frac{t}{2}
 \end{smallmatrix}\right] \right) =
\overline{h} (\cos t, -\alpha_1\sin t,  -\alpha_2\sin t), \label{eq:geodesic_special}\\ 
\frac{d}{dt} h (e^{t\matr{\Omega}}) &=
-2 \begin{bmatrix}\sin t & \alpha_1 \cos t &\alpha_2 \cos t \end{bmatrix}
\matr{\Gamma} \begin{bmatrix} \cos t & -\alpha_1\sin t  & -\alpha_2\sin t \end{bmatrix}^{\T}, \notag
\end{align}
and thus by \cite[(5.32)]{Absil08:Optimization} 
\begin{equation}\label{eq:quad_form_hessian}
\RInner{\matr{\Omega}}{\HessAppl{h}{\matr{I}_2}{\matr{\Omega}}} =  \left. \left( \frac{d^2}{dt^2} h (e^{t\matr{\Omega}})\right)\right|_{t=0} =
\begin{bmatrix} \alpha_1 & \alpha_2 \end{bmatrix} \HessPS{i}{j}{\matr{U}}  \begin{bmatrix} \alpha_1 \\ \alpha_2 \end{bmatrix},
\end{equation}
Finally, note that the geodesic is horizontal (its derivative stays in the horizontal space), hence \eqref{eq:quad_form_hessian} is valid for the Hessian of $\widetilde{h}$.
\end{proof}

\begin{corollary}
If $\matr{I}_2$ is a local maximizer of $\hij{i}{j}{\matr{U}}$, then $\HessPS{i}{j}{\matr{U}} \preceq 0$.
\end{corollary}

\begin{remark}\label{eq-Dij-euivalence}
Denote $\matr{\Gamma}=\Gamij{i}{j}{\matr{U}}$.
Then $\HessPS{i}{j}{\matr{U}}$ is negative definite if and only if
\begin{equation*}
\Gamma_{11} >
\lambda_{\rm max}\left(\begin{bmatrix}
\Gamma_{22} & \Gamma_{23}\\
\Gamma_{23}  & \Gamma_{33}
 \end{bmatrix}\right).
\end{equation*}
If, in addition, $\ProjGrad{\hij{i}{j}{\matr{U}}}{\matr{I}_2} =0$, this is equivalent to saying that $\lambda_1(\Gamma) > \lambda_2(\Gamma)$ (i.e., the first two eigenvalues are separated) and  $\Gamma_{11} = \lambda_1(\Gamma)$. 
\end{remark}

\subsection{Complex conjugate forms and equivalence of the cost functions}
For  $\tens{A} \in \CC^{n\times \cdots \times n}$ of order $d$ and an integer $t$, $0 \le t \le d$, we define the corresponding homogeneous conjugate  form \cite{JianL16:characterizing} (a generalization of a homogeneous polynomial) as 
\begin{equation}\label{eq-g-A}
g_{\tens{A}, t}(\vect{u})  
=\tens{A}\contr{1} \vect{u}^{*} \cdots \contr{t} \vect{u}^{*} \contr{t+1} \vect{u} \cdots  \contr{d} \vect{u},
\end{equation}
i.e., the tensor contracted $t$ times with $\vect{u}^{*}$ and the remaining $d-t$ times with $\vect{u}$. 
Then it is easy to see that the cost functions \eqref{cost-fn-general} and \eqref{cost-fn-general-trace}  can be written as\footnote{similarly to  contrast functions  \cite{comon2001source,Como04:ijacsp}}
\begin{equation}\label{eq-sum-simple-functions}
f(\matr{U}) = \sum\limits_{k=1}^{n} \gamma(\vect{u}_k),
\end{equation}
where $\gamma(\vect{u})$ is one of the following options depending on the cost function:
\begin{align}
\gamma(\vect{u}) & = \sum\limits_{\ell=1}^L \alpha_{\ell} | g_{\tens{A}^{(\ell)}, t_\ell}(\vect{u}) |^2, \quad \text{or}\label{eq-sum-of-squares}\\
\gamma(\vect{u}) & = g_{\tens{B}, d}(\vect{u}), \quad\text{where }\tens{B} \text{ is Hermitian  in the sense of \eqref{eq:hermitian}}.\label{eq-hermitian-form}
\end{align}
Note that we call forms of type \eqref{eq-hermitian-form} \emph{Hermitian forms}.
The equivalence of \eqref{cost-fn-general} and \eqref{cost-fn-general-trace} is established by the following result.
\begin{lemma}\label{prop:equivalence_cost_functions}
When restricted to norm-one vectors $\vect{u}$, $\|\vect{u}\|=1$, a function  $\gamma(\vect{u})$ is a Hermitian form  \eqref{eq-hermitian-form} of order $2d$ if and only if it can be represented 
as \eqref{eq-sum-of-squares}  for tensors $\tens{A}^{(1)},\ldots,\tens{A}^{(L)}$ of orders $d_1,\ldots,d_L \le d$.
\end{lemma}
\Cref{prop:equivalence_cost_functions} is a rather straightforward generalization of the results of \cite[Proposition 3.5]{JianL16:characterizing}; still, we provide a proof in \Cref{sec:ml_proofs} for completeness, and also because our notation is slightly different from that of  \cite{JianL16:characterizing}.  

We conclude this subsection by showing how to find Wirtinger derivatives for forms \eqref{eq-g-A}. 
\begin{lemma}\label{lem-multilinear-form}
For a  form $g(\vect{u}) = g_{\tens{A}, t} (\vect{u})$  defined in \eqref{eq-g-A}, it holds that
\begin{equation*}
\begin{split}
\deriv{g}{\vect{u}^*}(\vect{u}) &= \sum\limits_{k=1}^{t} 
\tens{A}\contr{1} \vect{u}^{*} \cdots \, \xcancel{\contr{k} \vect{u}^*} \, \cdots \contr{t} \vect{u}^{*} \contr{t+1} \vect{u} \cdots\cdots\cdots\cdots  \contr{d} \vect{u}, \\
\deriv{g}{\vect{u}}(\vect{u}) &= \sum\limits_{k=1}^{d- t} 
\tens{A}\contr{1} \vect{u}^{*} \cdots\cdots\cdots\cdots \contr{t} \vect{u}^{*} \contr{t+1} \vect{u} \cdots \xcancel{\contr{t+k} \vect{u}}  \cdots  \contr{d} \vect{u},\\
\deriv{|g(\vect{u}) |^2}{\vect{u}^{*}}& = (g(\vect{u}))^{*} \deriv{g}{\vect{u}^*}(\vect{u})  + (g(\vect{u}))  \left(\deriv{g}{\vect{u}}(\vect{u})\right)^*.  
\end{split}
\end{equation*}
\end{lemma}
\begin{proof}
The first two equations  follow by the rule of product differentiation and the following identities \cite[Table IV]{hjorundes2007complex-valued}
\begin{equation*}
\deriv{(\vect{u}^{\H}\vect{a})}{\vect{u}}(\vect{u})  = \deriv{(\vect{u}^{\T}\vect{a})}{\vect{u}^{*}}(\vect{u})  =  \vect{0},\quad \deriv{(\vect{u}^{\H}\vect{a})}{\vect{u}^{*}}(\vect{u}) = \deriv{(\vect{u}^{\T}\vect{a})}{\vect{u}}(\vect{u}) =  \vect{a}. \end{equation*}
The last equation  follows\footnote{An alternative proof can be derived by using the representation of  $|g(\vect{u}) |^2$ as a Hermitian form, contained in the proof of \Cref{lem-multilinear-form}.} from   the rule of differentiation of composition \cite[Theorem 1]{hjorundes2007complex-valued}, and the fact that $d|z|^2 = z^*dz + zdz^*$.
\end{proof}

\subsection{Riemannian gradients for cost functions of interest}\label{sec:cost_unified}
Before  computing derivatives for \eqref{cost-fn-general} and \eqref{cost-fn-general-trace}, we make a remark about symmetries in these  functions.
\begin{remark}\label{rem:partial_symmetry}
For any form $g_{\tens{A}, t} (\vect{u})$   \eqref{eq-g-A} we can assume without loss of generality that the tensor $\tens{A}$ is $t$-\emph{semi-symmetric}, i.e., satisfies the following symmetries:
\[
\tenselem{A}_{i_1\ldots i_t i_{t+1} \ldots i_{d}} = \tenselem{A}_{\pi_1(i_1\ldots i_t) \pi_2(i_{t+1} \ldots i_{d})}
\]
for any index $(i_1\ldots i_t i_{t+1} \ldots i_{d})$ and any pair of permutations $\pi_1(\cdot)$ and $\pi_2(\cdot)$ of indices corresponding to the same group of contractions in \eqref{eq-g-A}.
For example,
\begin{itemize}
\item 
for any  $\tens{A} \in \CC^{n\times n \times n}$  we can define  $\tenselem{T}_{ijk} = \frac{\tenselem{A}_{ijk}+\tenselem{A}_{ikj}}{2}$ such that
\begin{equation*}
\tens{A} \contr{1} \vect{u}^* \contr{2} \vect{u} \contr{3} \vect{u} = \tens{T} \contr{1} \vect{u}^* \contr{2} \vect{u} \contr{3} \vect{u};
\end{equation*}
\item similarly, for any $\tens{B} \in \CC^{n\times n \times n \times n}$  and $\tenselem{S}_{ijkl} = \frac{\tenselem{B}_{ijkl}+\tenselem{B}_{ijlk} + \tenselem{B}_{jikl}+\tenselem{B}_{jilk}}{4}$  we have
\begin{equation*}
\tens{B} \contr{1} \vect{u}^*  \contr{2} \vect{u}^* \contr{3} \vect{u} \contr{4} \vect{u} = \tens{S} \contr{1} \vect{u}^*  \contr{2} \vect{u}^* \contr{3} \vect{u} \contr{4} \vect{u}. 
\end{equation*}%
\end{itemize}%
Thus the tensors can be assumed to be semi-symmetric in  \eqref{cost-fn-general} and \eqref{cost-fn-general-trace}.
\end{remark}

Next, we are going to find  Riemannian gradients for cost functions \eqref{cost-fn-general} and \eqref{cost-fn-general-trace}.
Since the cost functions can be written as \eqref{eq-sum-simple-functions}, we have 
\begin{equation}\label{eq:gradient_contrast_like}
\deriv{f}{\matr{U}^{*}}(\matr{U}) = \begin{bmatrix} \deriv{\gamma}{\vect{u}^{*}}(\vect{u}_1) &
\cdots  & \deriv{\gamma}{\vect{u}^{*}}(\vect{u}_n)\end{bmatrix},
\end{equation}
hence    \Cref{lem-multilinear-form} can be used to prove the following result.

\begin{proposition}\label{prop:gradient-general-fn}
\begin{enumerate}[(i)]
\item Let $\tens{B}$ be a Hermitian $d$--semi-symmetric (as in \Cref{rem:partial_symmetry}) tensor. Then for the cost function \eqref{cost-fn-general-trace},
\begin{align}
&\left(\matr{U}^{\H} \deriv{f}{\matr{U}^{*}}(\matr{U}) \right)_{ij} = d \tenselem{V}_{ij\cdots j}, \quad \Lambda(\matr{U})_{ij}  = d (\tenselem{V}_{ij\cdots j} - \tenselem{V}_{i\cdots ij}),\label{cost-general-trace-gradient} \\
&\text{where }\tens{V} = \tens{B} \contr{1} \matr{U}^{\H} \cdots \contr{d} \matr{U}^{\H}\contr{d+1} \matr{U}^{\T} \cdots \contr{2d} \matr{U}^{\T} \label{eq-rotated-hermitian-tensor}
\end{align}
is the rotated Hermitian tensor.
\item Let   $\tens{A}$ be a $t$-semi-symmetric tensor and $\gamma(\vect{u}) = |g_{\tens{A}, t}(\vect{u}) |^2$. For  $f$ defined as \eqref{eq-sum-simple-functions} the gradients can be expressed  as 
\begin{equation}\label{cost-general-gradient}
\left(\matr{U}^{\H} \deriv{f}{\matr{U}^{*}}(\matr{U}) \right)_{ij}   = 
t \tenselem{W}^{*}_{j\ldots j}\tenselem{W}_{ij\ldots j} +
 (d-t)\tenselem{W}_{j\ldots j} \tenselem{W}^{*}_{j\ldots ji};
\end{equation}
\[
\Lambda_{ij}(\matr{U}) =  t (\tenselem{W}^{*}_{j\ldots j}\tenselem{W}_{ij\ldots j} - \tenselem{W}_{i\ldots i}\tenselem{W}^*_{ji\ldots i}) + (d-t)(\tenselem{W}_{j\ldots j} \tenselem{W}^{*}_{j\ldots ji} - \tenselem{W}^{*}_{i\ldots i} \tenselem{W}_{i\ldots ij} ), \]
where $\tens{W} = \tens{A}\contr{1} \matr{U}^{\H} \cdots \contr{t} \matr{U}^{\H} \contr{t+1} \matr{U}^{\T} \cdots  \contr{d} \matr{U}^{\T}$ is the rotated tensor.
\end{enumerate}
\end{proposition}
\begin{proof}
\begin{enumerate}[(i)]
\item By  \eqref{eq:gradient_contrast_like} and \Cref{lem-multilinear-form},  we get
\begin{align*}
& \left(\matr{U}^{\H} \deriv{f}{\matr{U}^{*}}(\matr{U}) \right)_{ij} = \vect{u}_i^{\H} \deriv{\gamma}{\vect{u}^{*}}(\vect{u}_j) \\
&= \sum\limits_{k=1}^{d} 
\tens{B}\contr{1} \vect{u}^{*}_j \cdots \, \xcancel{\contr{k} \vect{u}^*_j} \, \cdots \contr{d} \vect{u}^{*}_j \contr{d+1} \vect{u}_j \cdots  \contr{2d} \vect{u}_j \contr{k} \vect{u}^*_i  = d \tenselem{V}_{ij\ldots j},
\end{align*}
where the last equality is due to symmetries. The form of $\matr{\Lambda}$ follows from \eqref{eq:Lambda}.
\item The proof is similar\footnote{The proof can be also directly obtained from \eqref{eq-g-A-squared} and (i); the tensor needs to be semi-symmetrized before applying (i), hence the second term appears in \eqref{cost-general-trace-gradient} compared with \eqref{cost-general-gradient}.} to (i), and follows from \Cref{lem-multilinear-form}   and the equalities
\[
g(\vect{u}_j) = g_{\tens{A}, t}(\vect{u}_j) = \tenselem{W}_{j\ldots j}, \; \vect{u}^{\H}_i\deriv{g}{\vect{u}^*}(\vect{u}_j) = \tenselem{W}_{ij\ldots j},\;
\vect{u}^{\H}_i \left(\deriv{g}{\vect{u}}(\vect{u})\right)^* = \tenselem{W}^*_{ji\ldots i}
\]
\end{enumerate}
\end{proof}
\begin{remark}
Part ii of \Cref{prop:gradient-general-fn} also allows us  to find the Riemannian gradient for all functions of the form \eqref{cost-fn-general}, by summing  individual gradients for each $\tens{A}^{(\ell)}$.
For example,  the Riemannian gradient  of the cost function \eqref{eq-cost-jade}   simplifies to 
\begin{equation}\label{eq:lambda_jade}
\Lambda_{ij}(\matr{U}) =  \sum\limits_{\ell=1}^{L} ({W}^{(\ell)}_{jj}-W^{(\ell)}_{ii})^*{W}^{(\ell)}_{ij} + ({W}^{(\ell)}_{jj} -W^{(\ell)}_{ii}) ({W}^{(\ell)}_{ji})^{*},
\end{equation}
where ${W}^{(\ell)}$ is as in \Cref{example-quadratic-form-jade}. Note that \eqref{eq:lambda_jade} agrees with   \Cref{prop:dir_deriv}. 
\end{remark}

\subsection{Elementary update for Hermitian forms}\label{sec:elem_updates_general}
In this subsection, for simplicity, we only consider Hermitian tensors \eqref{eq:hermitian} of order $2d$ which we assume to be  $d$-semi-symmetric;
we also take $\tens{V}$ as in \eqref{eq-rotated-hermitian-tensor}.
Then $\hij{i}{j}{\matr{U}}(\matr{\Utwo})$ has the form
\[
\hij{i}{j}{\matr{U}}(\matr{\Utwo}) = \trace{\tens{V} \contr{1} \matr{G}^{\H} \cdots \contr{d} \matr{G}^{\H}\contr{d+1} \matr{G}^{\T} \cdots \contr{2d} \matr{G}^{\T}}
\]
where $\matr{G}= \Gmat{i}{j_k}{\matr{\Utwo}}$ is the Givens transformation.
Note that the Givens transformations change only elements of $\tens{V}$ with at least one of indices equal to $i$ or $j$, hence
\begin{equation}\label{eq:hij_to_2x2}
\hij{i}{j}{\matr{U}}(\matr{\Utwo})  = \underbrace{\sum\limits_{k\neq i,j}\tenselem{V}_{k\cdots k}}_{\text{constant}} +
\trace{\tens{T} \contr{1} \matr{\Utwo}^{\H} \cdots \contr{d} \matr{\Utwo}^{\H}\contr{d+1} \matr{\Utwo}^{\T} \cdots \contr{2d} \matr{\Utwo}^{\T}},
\end{equation}
where $\tens{T} = \tens{V}_{(i,j),\ldots,(i,j)}$ is the $2\times \cdots \times  2$ subtensor of $\tens{V}$ corresponding to indices $i,j$.
Then the following result characterizes the elementary rotations.

\begin{theorem}\label{thm:elementary_updates}
Let $\tens{T}$ be a Hermitian $2d$-order $d$-semi-symmetric $2\times \cdots \times 2$ tensor. 
Then there exists a $3\times \cdots \times 3$ real symmetric tensor $\GamTens$ of order $2m$ for $m=\lfloor \frac{d}{2} \rfloor$ such that
\[
\widetilde{h}(c,s_1,s_2) \eqdef \trace{\tens{T} \contr{1} \matr{\Utwo}^{\H} \cdots \contr{d} \matr{\Utwo}^{\H}\contr{d+1} \matr{\Utwo}^{\T} \cdots \contr{2d} \matr{\Utwo}^{\T}} = 
\GamTens \contr{1} \vect{r}  \cdots \contr{2m} \vect{r},
\]
where  $\matr{\Utwo} = \matr{\Utwo}(c, s_1, s_2)  $ and $\vect{r} = \vect{r}(c,s_1,s_2)$ are as in \eqref{unitary-para-2} and \eqref{eq:w_definition}.
\end{theorem}
The proof of \Cref{thm:elementary_updates} is contained in \Cref{sec:ml_proofs}.

 \begin{remark}
 Theorem~\ref{thm:elementary_updates} implies that:
 \begin{itemize}
 \item $m = 1$ for  $d \le 3$, i.e., $\GamTens$ is a symmetric $3\times 3$ matrix (called $\matr{\Gamma}$ in \Cref{lem-cost-quadratic-form}).
 Thus,  Theorem~\ref{thm:elementary_updates}  provides a proof for \Cref{lem-cost-quadratic-form}.
 \item $m=2$ for $d=4$, in particular, the elementary update for the $4$-th order complex tensor diagonalization requires maximizing a $4$-th order ternary form (which was established in \cite{de1997signal} for this particular case).
 \item For  $d >3$  (unlike $d\le 3$), the update cannot be computed in a closed form.
 \end{itemize}
 \end{remark}

\begin{remark}
The proof of \Cref{thm:elementary_updates} gives a systematic way to find the coefficients of $\GamTens$ for any instance of \eqref{cost-fn-general} or \eqref{cost-fn-general-trace}, and thus generalizes existing expressions derived for special cases (see \cite[Ch. 5]{comon2010handbook}).
\end{remark}

\section{Weak convergence results}\label{sec:weak}
\subsection{Global rates of convergence of descent algorithms on  manifolds}
We first recall a simplified version of  result presented in \cite[Thm. 2.5]{boumal2016global} on convergence of  ascent  algorithms (originally proposed in \cite{boumal2016global}  for retraction-based algorithms).

\begin{lemma}[{\cite[{Theorem 2.5}]{boumal2016global}}]\label{Riemannian-descent-convergence}
Let $f:\mathcal{M}\rightarrow\RR$ be bounded from above  by $f^{*}$.
Suppose that, for a sequence\footnote{Note that in the original formulation of {\cite[{Theorem 2.5}]{boumal2016global}} $x_k$ were chosen as retractions of some vectors in $\mathbf{T}_{x_{k-1}}$. However, it is easy to see that this condition is not needed in the proof.} of $x_k$,  there exists $c>0$ such that  
\begin{equation}\label{Riemannian-descent-condition}
f(x_{k +1}) - f(x_{k}) \ge c\|\ProjGrad{f}{x_k}\|^2.
\end{equation}
Then 
\begin{enumerate}[(i)]
\item $\|\ProjGrad{f}{x_k}|\| \to  0$ as $k\to \infty$;
\item We can find an $x_k$ with $\|\ProjGrad{f}{x_k}\|  \le  \varepsilon$  and $f(x_{k})  \ge  f(x_0)$ in at most
\begin{equation*}
K_{\varepsilon} = \left\lceil \frac{f^* - f(x_0)}{c} \frac{1}{\varepsilon^2} \right\rceil 
\end{equation*}
iterations; 
i.e., there exists $k \le K_{\varepsilon}$ such that $\|\ProjGrad{f}{x_k}\| < \varepsilon$.
\end{enumerate}
\end{lemma}
\begin{proof}
\begin{enumerate}[(i)]
\item
We use the classic telescopic sums argument to obtain
\[
f^* -  f(x_0)   \ge  f(x_K)  - f(x_0) =  \sum\limits_{k=0}^{K-1} (f(x_{k+1}) - f(x_k)) \ge c \sum\limits_{k=0}^{K-1} \|\ProjGrad{f}{x_k}\|^2.
\] 
Then we have that  $\sum\limits_{k=0}^{\infty} \|\ProjGrad{f}{x_k}\|^2$ is convergent, thus  $\|\ProjGrad{f}{x_k}\| \to 0$. 
\item 
Assume that $ \|\ProjGrad{f}{x_k}\| > \varepsilon$ for all  $K-1$ iterations. Then, in a similar way, 
\begin{equation*}
f^* -  f(x_0)    \ge c K \min_{0\le k \le K-1} \|\ProjGrad{f}{x_k}\|^2 >  c K \varepsilon^2, 
\end{equation*}
which can only hold if $K < K_\varepsilon$.
\end{enumerate}
\end{proof}
For checking  the  ascent  condition \eqref{Riemannian-descent-condition},  we recall a lemma on retractions. 
\begin{definition}\label{dfn-retraction}(\cite[Definition 4.4.1]{Absil08:Optimization})
A \emph{retraction} on a manifold $\mathcal{M}$ is a smooth mapping ${\rm Retr}$
from the tangent bundle $\TangBundle{\mathcal{M}}$ to $\mathcal{M}$ with the following properties.
Let ${\rm Retr}_{x}:\textbf{T}_{x}\mathcal{M}\rightarrow\mathcal{M}$ denote the restriction of
${\rm Retr}$ to the tangent vector space $\textbf{T}_{x}\mathcal{M}$.\\
(i) ${\rm Retr}_{x}(\vect{0}_{x}) = x$, where $\vect{0}_{x}$ is the zero vector in $\textbf{T}_{x}\mathcal{M}$;\\
(ii) The differential of ${\rm Retr}_{x}$ at $\vect{0}_x$, ${\rm DRetr}_{x}(\vect{0}_x)$, is the identity map.
\end{definition}

\begin{lemma}[{\cite[{Lemma 2.7}]{boumal2016global}}]\label{lemma-lipschitz-gradient}
Let $\mathcal{M}\subseteq\RR^{n}$ be a compact Riemannian submanifold.
Let ${\rm Retr}$ be a retraction on $\mathcal{M}$.
Suppose that $f:\mathcal{M}\rightarrow\RR$ has Lipschitz continuous gradient in the convex hull of $\mathcal{M}$.
Then there exists
$L\geq 0$ such that for all $x\in\mathcal{M}$ and  $\eta\in \mathbf{T}_{x}\mathcal{M}$,
it holds that
\begin{equation}\label{eq-1st-taylor}
\left|{f}({\rm Retr}_{x}(\eta))-\big(f(x)+\langle\eta,\ProjGrad{f}{x}\rangle\big)\right|\leq \frac{L}{2}\|\eta\|^2,
\end{equation}
i.e., ${f}({\rm Retr}_{x}(\eta))$ is uniformly well approximated by its first order approximation.
\end{lemma}

\begin{corollary}\label{lipschitz-gradient-Un}
Let $f$ be any of the functions \eqref{cost-fn-general} or \eqref{cost-fn-general-trace}.
Then there exists a constant $L\geq 0$  such that the uniform (on $\UN{n}$) approximation bound holds true.
\end{corollary}
\begin{proof}
Note that we can view $\UN{n}$ as a real submanifold of $\CC^{n\times n}$, and its convex hull is compact.
The cost functions  \eqref{cost-fn-general} and \eqref{cost-fn-general-trace} are defined on $\CC^{n\times n}$ and are polynomial in the real and imaginary parts of $\matr{U}$.
This implies Lipschitz continuity of $f$ on the convex hull of $\UN{n}$, hence \Cref{lemma-lipschitz-gradient} can be applied.
\end{proof}

\subsection{Convergence of Jacobi-G algorithm to stationary points}
We will show in this subsection that the iterations in \Cref{Jacobi-G-general} are a special case of the iterations in \Cref{Riemannian-descent-convergence}, and the convergence results of  \Cref{Riemannian-descent-convergence} apply.

\begin{proposition}\label{theorem-local-gene}
Let $f:\UN{n}\rightarrow\RR$  be  one of the functions \eqref{cost-fn-general} or \eqref{cost-fn-general-trace},  and $L\geq 0$ be from \Cref{lipschitz-gradient-Un}.  For \Cref{Jacobi-G-general}, we have:
\begin{enumerate}[(i)]
\item $\|\ProjGrad{f}{\matr{U}_k}\|\rightarrow 0$ in Algorithm~\ref{Jacobi-G-general}; in particular, every accumulation point in Algorithm \ref{Jacobi-G-general} is a stationary point. 
\item  For $\delta$ as in \eqref{inequality-gra-based}, Algorithm~\ref{Jacobi-G-general} needs at most  
\begin{equation*}
\left\lceil \frac{2L (f^* - f(x_0))}{\delta^2} \frac{1}{\varepsilon^2} \right\rceil 
\end{equation*}
iterations  to reach an $\varepsilon$-optimal solution ($\|\ProjGrad{f}{\matr{U}_k}\| \le \varepsilon$).
\end{enumerate}
\end{proposition}
\begin{proof}
We need to show that the ascent conditions are satisfied.
Let $h = \hij{i}{j}{\matr{U}}$ be as in \eqref{eq-func-h-new} and $\matr{\Utwo}_{\rm opt}$  be its maximizer. 
We set
\begin{equation*}
\matr{\Delta} = \matr{U} \mathcal{P}_{i,j}^{\T} \mathcal{P}_{i,j} (\matr{\Lambda}(\matr{U}))\in  \mathbf{T}_{\matr{U}}\UN{n},
\end{equation*}
which is a projection of $\ProjGrad{f}{\matr{U}}$ onto the tangent space to the submanifold of the matrices of type $\matr{U}\Gmat{i}{j}{\matr{\Utwo}}$.
 Next, denote  $\matr{\Psi}_1 =  \text{Exp}_{\matr{I_2}}(\frac{1}{L}\ProjGrad{h}{\matr{I}_2})$.
Then, by \Cref{lemma-lipschitz-gradient} and \Cref{lipschitz-gradient-Un}, we have\footnote{Note that the exponential map \eqref{eq:exp_map} is a retraction (see  \cite[Proposition 5.4.1]{Absil08:Optimization}).} that
\begin{equation*}
\begin{split}
h(\matr{\Utwo}_{\rm opt}) - h(\matr{I}_2) & \ge 
 h(\matr{\Utwo}_{1}) - h(\matr{I}_2)  = f\left(\text{Exp}_{\matr{U}}\left(\frac{\matr{\Delta}}{L} \right)\right) - f(\matr{U}) \\
&\ge \RInner{\frac{\matr{\Delta}}{L}}{\ProjGrad{f}{\matr{U}}} - \frac{L}{2} \left\|\frac{\matr{\Delta}}{L}\right\|^2 =
\frac{\|\ProjGrad{h}{\matr{I}_2}\|^2}{2},
\end{split}
\end{equation*}
where  the last equality is from \eqref{eq-riem-grad} and \eqref{eq-gradient-h}.  
Finally, we note that
\begin{equation*}
f(\matr{U}_{k}) - f(\matr{U}_{k-1}) = h_k(\matr{\Psi}_k)  - h_k(\matr{I}_2) \ge
\frac{1}{2L} \|\ProjGrad{h_k}{\matr{I}_2}\|^2 \geq \frac{\delta^2}{2L}
\|\ProjGrad{f}{\matr{U}_{k-1}}\|^2,
\end{equation*}
and thus the descent condition \eqref{Riemannian-descent-condition} holds with the constant $\frac{\delta^2}{2L}$.
\end{proof}

\section{\L{}ojasiewicz   inequality}\label{sec:Lojasiewicz}
In this section, we recall known results and preliminaries that are needed for the main results in \Cref{sec-main-results}.

\subsection{\L{}ojasiewicz gradient inequality and speed of convergence}
Here we recall the results on convergence of descent algorithms on analytic submanifolds that use \L{}ojasiewicz gradient inequality \cite{Loja63:cnrs}, as presented in \cite{Usch15:pjo}.
These results were used in \cite{LUC2018} to prove the global convergence of Jacobi-G on the orthogonal group. 

\begin{definition}[\L{}ojasiewicz gradient inequality,  {\cite[Definition 2.1]{SU15:pro}}]\label{def:Lojasiewicz}
Let $\mathcal{M} \subseteq \RR^n$ be a Riemannian submanifold of $\RR^n$.
The function $f: \mathcal{M} \to \RR$  satisfies  a  \emph{\L{}ojasiewicz  gradient inequality} at a point  $x \in \mathcal{M}$, if there exist  $\delta>0$, $\sigma>0$ and $\zeta\in (0,\frac{1}{2}]$
such that for all $y\in\mathcal{M}$ with $\|y-x\|<\delta$,
it holds that
\begin{equation}\label{eq:Lojasiewicz}
|{f}(x)-{f}(y)|^{1-\zeta}\leq \sigma\|\ProjGrad{f}{y}\|.
\end{equation}
\end{definition}

The following lemma guarantees that \eqref{eq:Lojasiewicz} is satisfied  for the real analytic functions defined on an analytic manifold.  
\begin{lemma}[{\cite[Proposition 2.2 and Remark 1]{SU15:pro}}]\label{lemma-SU15}
Let $\mathcal{M}\subseteq\RR^n$ be an analytic submanifold\footnote{See {\cite[Definition 2.7.1]{krantz2002primer}} or \cite[Definition 5.1]{LUC2018} for a definition of an analytic submanifold.} and  $f: \mathcal{M} \to \RR$ be a real  analytic function.
Then  for  any  $x\in \mathcal{M}$,  $f$ satisfies a \L{}ojasiewicz  gradient  inequality \eqref{eq:Lojasiewicz}  for  
some\footnote{The values of $\delta,\sigma,\zeta$ depend on a specific point.} $\delta,\sigma>0$ and $\zeta\in (0,\frac{1}{2}]$.
\end{lemma}

\L{}ojasiewicz  gradient  inequality allows for proving convergence of optimization algorithms  to a single limit point. 
\begin{theorem}[{\cite[Theorem 2.3]{SU15:pro}}]\label{theorem-SU15}
Let $\mathcal{M}\subseteq\RR^n$ be an analytic submanifold
and 
$\{x_k:k\in\NN\}\subseteq\mathcal{M}$.
Suppose that $f$ is real analytic and, for large enough $k$,\\
(i) there exists $\sigma>0$ such that
\begin{equation}\label{eq:sufficient_descent}
|{f}(x_{k+1})-{f}(x_k)|\geq \sigma\|\ProjGrad{f}{x_k}\|\|x_{k+1}-x_{k}\|;
\end{equation}
(ii) $\ProjGrad{f}{x_k}=0$ implies that $x_{k+1}=x_{k}$.\\
Then any accumulation point $x_*$ of $\{x_k:k\in\NN\}\subseteq\mathcal{M}$ is the only limit point.

If, in addition, for some $\kappa > 0$  and for large enough  $k$  it holds that 
\begin{equation}\label{eq:safeguard}
\|x_{k+1} - x_{k}\| \ge \kappa \|\ProjGrad{f}{x_k}\|,
\end{equation}
then the following convergence rates apply
\begin{equation*}
\|x_{k} - x^*\| \le  C\begin{cases}
e^{-ck}, & \text{ if } \zeta = \frac{1}{2} \text{ (for some }c > 0\text{)}, \\
k^{-\frac{\zeta}{1-2\zeta}}, & \text{ if } 0 < \zeta < \frac{1}{2},
\end{cases}
\end{equation*}
where $\zeta$ is the parameter in \eqref{eq:Lojasiewicz} at the limit point $x_*$. 
\end{theorem}

\begin{remark}\label{rem:sufficient_descent_relax}
We can  relax  the conditions of \Cref{theorem-SU15}  as follows.  We can require just that \eqref{eq:sufficient_descent} holds  for all $k$ such that  $\|x_k -x_*\| < \varepsilon$, where $x_*$ is  an accumulation point of the sequence and $\varepsilon > 0$ is some radius. 
This can be verified by inspecting the proof of \Cref{theorem-SU15} (see also the proof of \cite[Theorem 3.2]{AbsMA05:sjo})
\end{remark}

In the case $\zeta = \frac{1}{2}$, according to \Cref{theorem-SU15}, the convergence is linear (similarly to the classic results  on local convergence of the gradient descent algorithm  \cite{Poly63:gradient,BoydV04}).
In the optimization literature, the inequality  \eqref{eq:Lojasiewicz} with $\zeta = \frac{1}{2}$ is often called \emph{Polyak-\L{}ojasiewicz inequality}\footnote{The inequality \eqref{eq:Lojasiewicz} with $\zeta = \frac{1}{2}$ goes back to Polyak \cite{Poly63:gradient}, who used it for proving  linear convergence of the gradient descent algorithm.}.
In the next subsection, we recall some sufficient conditions  for Polyak-\L{}ojasiewicz inequality to hold.

\subsection{\L{}ojasiewicz inequality at stationary  points}
It is known, and widely used in optimization (especially in the Euclidean case), that around a strong local maximum the function satisfies the Polyak-\L{}ojasiewicz inequality.
In fact, it is also valid for non-degenerate stationary  points, as shown in  \cite{Hu2018}.
Here we recall the most general recent result on possibly degenerate stationary  points that satisfy the so-called Morse-Bott property (see also \cite[p.248]{bott1954}).

\begin{definition}[{\cite[Definition 1.5]{feehan2018optimal}}]\label{def-local-morse-bott}
Let $\mathcal{M}$ be a $C^\infty$ submanifold and  $f: \mathcal{M} \to \RR$ be a $C^{2}$ function. Denote the set of stationary points as  
$${\rm Crit} f = \{x \in \mathcal{M} : \ProjGrad{f}{x} = 0\}.$$
The function $f$ is said to be \emph{Morse-Bott} at $x_0\in\mathcal{M}$  if there exists an open neighborhood $\set{U} \subseteq \mathcal{M}$ of $x_0$ such that
\begin{enumerate}[(i)]
\item  $\set{C} = \set{U} \cap {\rm Crit}f$ is a relatively open, smooth submanifold of $\mathcal{M}$;
\item $\mathbf{T}_{x_0}\set{C}  = {\rm Ker}\ \Hess{f}{x_0} $.
\end{enumerate}
\end{definition}
\begin{remark}\label{remMorseBott}
(i) If $x_0\in\mathcal{M}$ is a non-degenerate stationary point, then $f$ is  Morse-Bott at $x_0$, since $\{x_0\}$ is a zero-dimensional manifold in this case.\\
(ii) If $x_0\in\mathcal{M}$ is a degenerate stationary point, then condition (ii) in \Cref{def-local-morse-bott}  can be rephrased\footnote{due to the fact that $\mathbf{T}_{x_0}\mathcal{C}  \subseteq {\rm Ker}\ \Hess{f}{x_0}$.} as 
\begin{equation}\label{eq:Hessian_rank_condition}
\rank{\Hess{f}{x_0}} =  \dim \mathcal{M} - \dim \set{C}.
\end{equation}
\end{remark}
For the  functions that satisfy the  Morse-Bott property, it was recently shown that the Polyak-\L{}ojasiewicz inequality holds true.
\begin{theorem}[{\cite[Theorem 3, Corollary 5]{feehan2018optimal}}]\label{thm:PLMorseBottEuclidean}
If $\set{U}\subseteq\RR^{n}$ is an open  subset and $f: \set{U} \to \RR$ is Morse-Bott at a stationary point $x$,  then there exist $\delta, \sigma>0$  such that
\begin{equation*}
|{f}(y)-{f}(x)| \leq \sigma \|\nabla{f}(y)\|^2,
\end{equation*}
for any $y\in \set{U}$ satisfying $\|y-x\| \le \delta$.
\end{theorem}

We can also easily deduce the same result  on a smooth manifold $\mathcal{M}$.

\begin{proposition}\label{cor:PLMorseBottManifold}
If $\set{U}\subseteq\mathcal{M}$ is an open subset and a $C^2$ function $f: \set{U} \to \RR$ is Morse-Bott at a stationary point $x$,  then there exist an open neighborhood $\set{V} \subseteq \set{U}$ of $x$ and $\sigma>0$  such that for all $y \in \set{V}$ it holds that
\begin{equation*}
|{f}(y)-{f}(x)| \leq \sigma \|\ProjGrad{f}{y}\|^2.
\end{equation*}
\end{proposition}

\begin{proof}
Consider the exponential map ${\rm Exp}_x: \mathbf{T}_x {\mathcal{M}} \to \mathcal{M}$, which is a local diffeomorphism. Let $\set{W}\subseteq \mathbf{T}_x {\mathcal{M}}$ be an open subset such that ${\rm Exp}_x(\set{W})=\set{U}$. 
Let $\widehat{f} = f \circ {\rm Exp}_x$ be the composite map from $\set{W}$ to $\RR$. Then
\begin{equation}
\nabla \widehat{f} (y') = \mathbf{J}^{\T}_{{\rm Exp}_x} (y') \ProjGrad{f}{y} ,
\end{equation}
where $y'\in \set{W}$ and $y = {\rm Exp}_x(y')$.
It follows that ${\rm Exp}_x$ gives a diffeomorphism between ${\rm Crit}f$ and ${\rm Crit}\widehat{f}$.
Since $\Hess{f}{x} = {\rm H}_{\widehat{f}} (0)$ by \cite[Proposition 5.5.5]{Absil08:Optimization}, we have that $\widehat{f}$ is Morse-Bott at $0$.
Therefore, by \Cref{thm:PLMorseBottEuclidean}, there exist $\sigma'>0$, $\sigma>0$ and an open neighborhood $\set{V} \subseteq \set{U}$ of $x$ such that 
\begin{equation*}
|{f}(y)-{f}(x)| = |\widehat{f}(y')-\widehat{f}(0)| \leq \sigma' \|\nabla \widehat{f} (y')\|^2 \le  \sigma \|\ProjGrad{f}{y}\|^2,
\end{equation*}
for any $y\in \set{V}$, 
where the last inequality holds because $\mathbf{J}_{{\rm Exp}_x}$ is nonsingular in a neighborhood of $x$.
\end{proof}

\begin{remark}
For the case of non-degenerate stationary  points and $C^\infty$ functions, \Cref{cor:PLMorseBottManifold} is proved in \cite[Lemma 4.1]{Hu2018}, which is a simple corollary of Morse Lemma  \cite[Lemma 2.2]{milnor1963morse}.
For $C^{\infty}$ functions and Morse-Bott functions,  \Cref{cor:PLMorseBottManifold}  (as noted in \cite{feehan2018optimal}) is also a simple corollary of Morse-Bott Lemma \cite{banyaga2004}.
\end{remark}

\begin{remark}
Morse-Bott property is known to be useful for studying convergence properties.
For example, it is shown in \cite[Appendix C]{HelmM94:optimization} that if the cost function is (globally) Morse-Bott, i.e., satisfies the Morse-Bott property at all the stationary point, then the continuous gradient flow converges to a single point.
\end{remark}

Finally, we recall an important property of non-degenerate local maxima, which follows from the classic Morse Lemma
\cite{milnor1963morse}.

\begin{lemma}\label{lem:morse_maximum}
Let $x$ be a non-degenerate local maximum (according to  \Cref{def:nondegenerate})  of a smooth function $f$ such that $f(x) = c$.
Then there exists a simply connected open neighborhood $\set{W}$ of $x$ such that 
\begin{itemize}
\item $x$ is the only critical point in $\set{W}$;
\item its boundary is a level curve (i.e $f(y) = a < c$, for all $y \in \delta(\set{W})$;
\item the superlevel sets ${\set{W}}_b =  \{x \in {\set{W}}, {f}(x) \ge b > a \}$ are simply connected and nested.
\end{itemize}
\end{lemma}

\begin{remark}\label{rem:negative_deriv_geodesic}
In \Cref{lem:morse_maximum}, we can also select the neighborhood in such a way that Hessian is negative definite at each point $y$, which implies that  for any geodesic\footnote{A related discussion on geodesic convexity of functions can be found in  \cite{rapcsak1991}.} $\gamma(t)$ passing through $y$, $\gamma(0) = y$, 
the second derivative of $f(\gamma(t))$ at $0$ is  negative.
\end{remark}

\section{Convergence results  based on \L{}ojasiewicz inequality}\label{sec-main-results}
\subsection{Preliminary lemmas: checking the decrease conditions}\label{sec-decrease-conditions}
In this subsection, we are going to find some sufficient  conditions for \eqref{eq:sufficient_descent} and \eqref{eq:safeguard} to hold in \Cref{Jacobi-G-general}, which will  allow us to use \Cref{theorem-SU15}.

Let $\matr{U}_{k} = \matr{U}_{k-1} \Gmat{i_k}{j_k}{\matr{\Utwo}_k}$ be the iterations in \Cref{Jacobi-G-general}.  Obviously,
\begin{equation*}
 \|\matr{U}_{k}- \matr{U}_{k-1}\|  = \| \matr{\Utwo}_k-\matr{I}_2\|.
\end{equation*} 
Assume that $\matr{\Utwo}_k$ is obtained as in \Cref{lem-cost-quadratic-form}, i.e., by taking $\vect{w}$ as the leading eigenvector of $\Gamij{i_k}{j_k}{\matr{U}_{k-1}}$  (normalized so that $w_1 = \cos 2\theta =2c^2-1> 0$ in \eqref{eq:w_definition}) as in \Cref{remark-eigenvectors},  and retrieving $\matr{\Psi}_k$ from $\vect{w}$ according to \eqref{unitary-para-2} and  \eqref{eq:w_definition}.
We first express  $\|\matr{\Utwo}_k-\matr{I}_2\|$ through $w_1$. 

\begin{lemma}\label{lem:norm_equiv}
For the iterations  $\matr{\Utwo}_k$  obtained as in \Cref{lem-cost-quadratic-form}, it  holds that
\begin{equation}\label{eq-itera-control}
\sqrt{2}  \|\matr{\Utwo}_k- \matr{I}_{2}\| \ge \sqrt{1-w_1^2} \ge \frac{ \sqrt{ \sqrt{2}  +2} }{2}   \|\matr{\Utwo}_k- \matr{I}_{2}\|
\end{equation}
\end{lemma}
\begin{proof}
Note that
\begin{equation*}
 \|\matr{\Utwo}_{k}- \matr{I}_2\|   = \left\| \begin{bmatrix}
 c -1 & -s \\
  s^{\ast} & c-1
 \end{bmatrix}\right\| =  \sqrt{2 (1-c)^2  + 2(1-c^2)}  =   2 \sqrt{1-c}. 
\end{equation*}
Next, we note that $\sqrt{1-w_1^2} = {2}c \sqrt{1-c^2}$ and
\[
\frac{\sqrt{1-w_1^2}}{ \|\matr{\Utwo}_{k}- \matr{I}_2\|} = \frac{{2}c \sqrt{1-c^2}}{2 \sqrt{1-c}} = c \sqrt{1+c}.
\]
By \Cref{remark-eigenvectors}, we have $c \in [\frac{1}{\sqrt{2}}; 1]$, hence the ratio can be bounded from above by its values at the endpoints of the interval. 
\end{proof}

Since we are looking at \Cref{Jacobi-G-general}, we can replace in both inequalities of \eqref{eq-itera-control} $\ProjGrad{f}{\matr{U}_{k-1}}$ with
$\ProjGrad{\hij{i}{j}{\matr{U}}}{\uI_2}$ . 
Next, we prove a result for condition \eqref{eq:safeguard}. 

\begin{lemma}\label{lem:safeguard_Gamma}
Let $f: \UN{n} \to \RR$ be  as in \Cref{lem-cost-quadratic-form}. 
Then there exists a universal constant $\kappa>0$ such that 
\begin{equation*}
\|\matr{\Utwo}_k - \matr{I}_{2}\| \ge \kappa \|\ProjGrad{h_k}{\uI_2}\|.
\end{equation*}
\end{lemma}
\begin{proof}
We denote $\matr{\Gamma}=\Gamij{i_k}{j_k}{\matr{U}_{k-1}}$ as in \eqref{eq-cost-quadratic-form}. By \Cref{prop:dir_deriv}, we have that
\begin{equation*}
\|\ProjGrad{h_k}{\uI_2}\| = 2\sqrt{2} \sqrt{\Gamma_{12}^2 + \Gamma_{13}^2}.
\end{equation*}
By \Cref{lem:norm_equiv}, it is sufficient to prove that
\begin{equation*}
1-w_1^2 \ge \kappa' (\Gamma_{12}^2 + \Gamma_{13}^2)
\end{equation*}
for a universal constant $\kappa'>0$.  Let $\lambda_1\geq\lambda_2\geq\lambda_3$  be the eigenvalues of $\matr{\Gamma}$.
Without loss of generality, we set $\matr{\Gamma}' = \matr{\Gamma} - \lambda_3 \matr{I}_3$,
$\mu_1 = \lambda_1 - \lambda_3$ and $\mu_2 = \lambda_2 - \lambda_3$.
Then
\begin{equation}\label{eq:Gamma_EVD}
\matr{\Gamma}' =  \mu_1 \vect{w} \vect{w}^{\T} + \mu_2 \vect{v}\vect{v}^{\T},
\end{equation}
where $\vect{v}$ is the second eigenvector of $\matr{\Gamma}$.
It follows  that
\begin{equation}\label{eq:Gamma1213}
\begin{split}
\Gamma_{12}^2 + \Gamma_{13}^2 &= (\Gamma'_{12})^2 + (\Gamma'_{13})^2 = (\mu_1 w_1 w_2 + \mu_2 v_1 v_2)^2 + (\mu_1 w_1 w_3 + \mu_2 v_1 v_3)^2 \\
& =   \mu_1^2 w_{1}^2 (w_{2}^2 + w_3^2)  + 2\mu_1\mu_2 w_1 v_1 (w_2 v_2 + w_3 v_3)+ \mu^2_2 v_{1}^2 (v_{2}^2 + v_3^2) \\
& = \mu_1^2 w_{1}^2 (1- w_{1}^2) - 2\mu_1\mu_2 w_1^2 v_1^2 + \mu^2_2 v_{1}^2 (1- v_{1}^2) \\
&  \le (1-w_1^2) (\mu_1^2 + \mu_2^2), 
\end{split}
\end{equation}
where the last equality and inequality is due to orthonormality of $\vect{v}$ and $\vect{w}$  (which implies  $v_1^2 \le 1- w_1^2$).
By expanding $\mu_1$ and $\mu_2$, it is not difficult to verify that $\mu_1^2 + \mu_2^2 \le 2\|\Gamma\|^2$.
Finally, by \Cref{thm:elementary_updates}, the elements of  $\matr{\Gamma}$ continuously depend on $\matr{U}\in\UN{n}$.
Therefore, $\|\matr{\Gamma}\|$ is bounded from above, and thus the proof is completed. 
\end{proof}

We are ready to  check the sufficient decrease condition \eqref{eq:sufficient_descent}.
\begin{lemma}\label{lem:sufficient_decrease_Gamma}
Let $\matr{\Gamma}=\Gamij{i_k}{j_k}{\matr{U}_{k-1}}$ be as in \eqref{eq-cost-quadratic-form}. 
Let $\lambda_1\geq\lambda_2\geq\lambda_3$  be the eigenvalues of $\matr{\Gamma}$, and $\eta = \frac{\lambda_2 - \lambda_3}{\lambda_1 - \lambda_3}$.
Suppose that $1 -\eta \ge \varepsilon$ for some $\varepsilon>0$. Then
\begin{equation*}
|h_k(\matr{\Utwo}_k) - h_k(\matr{I}_2)|\ge \frac{\varepsilon}{4}\|\ProjGrad{h_k}{\matr{I}_2}\| \sqrt{1-w_1^2}.
\end{equation*}
\end{lemma}
\begin{proof}
Define the ratio
\begin{equation}\label{eq:ratio_eigenvectors}
q(\matr{\Gamma},\vect{w}) = \frac{(\vect{w}^{\T} \matr{\Gamma} \vect{w}  - \Gamma_{11})^2}{(\Gamma_{12}^2 + \Gamma_{13}^2) (1-w_1^2)}.
\end{equation}
It is sufficient to prove that $q(\matr{\Gamma},\vect{w})  \ge {\varepsilon^2}/{2}$.
Denote $\rho \eqdef 1- w_1^2 \ge v_1^2$, 
where $\vect{v}$ is as  in the proof of \Cref{lem:safeguard_Gamma}.
From  \eqref{eq:Gamma_EVD} and  \eqref{eq:Gamma1213} we immediately have  
\begin{align}
&\vect{w}^{\T} \matr{\Gamma} \vect{w}  - \Gamma_{11} = \mu_1 - (\mu_1 w_1^2 + \mu_2 v_1^2) = \mu_1(\rho - \eta v_1^2) \ge \mu_1 \rho (1- \eta), \label{eq-h-increase-01} \\
&\Gamma_{12}^2 + \Gamma_{13}^2   \le \rho \mu_1^2 (1+\eta^2). \label{eq-h-increase-02}
\end{align}
Using  \eqref{eq-h-increase-01} and \eqref{eq-h-increase-02}, we get
\begin{equation*}
\begin{split}
\frac{1}{q(\matr{\Gamma},\vect{w})} & =
 \frac{(\Gamma_{12}^2 + \Gamma_{13}^2)\rho}{(\vect{w}^{\T} \matr{\Gamma} \vect{w}  - \Gamma_{11})^2}
 \le \frac{\rho^2 \mu_1^2 (1+\eta^2)}{\rho^2 \mu_1^2 (1-\eta)^2}  \le \frac{2}{\varepsilon^2}. 
\end{split}
\end{equation*}
The proof is complete. 
\end{proof}

\subsection{Main results}\label{sec:results_local_linear}
\begin{theorem}\label{thm:accumulation_points}
Let  $f:\UN{n}\rightarrow\RR$ be as in \Cref{lem-cost-quadratic-form},
and $\overline{\matr{U}}$ be  an accumulation point of \Cref{Jacobi-G-general} (and $\ProjGrad{f}{\overline{\matr{U}}} = 0$  by \Cref{theorem-local-gene}).
Assume that $\HessPS{i}{j}{\overline{\matr{U}}}$  defined in \eqref{eq:hess_2x2} is  negative definite for all $i < j$. 
Then
\begin{enumerate}[(i)]
\item $\overline{\matr{U}}$ is the only limit  point and convergence rates in \Cref{theorem-SU15} apply.
\item  If the rank of Riemannian Hessian is maximal at $\overline{\matr{U}}$ (i.e., $\rank{\Hess{f}{\overline{\matr{U}}}} = n(n-1)$), then the speed of convergence is linear.
\end{enumerate}
\end{theorem}
\begin{proof}
\begin{enumerate}[(i)]
\item
Since  $\HessPS{i}{j}{\overline{\matr{U}}}$  is negative definite for any $i \neq j$,  the two top eigenvalues of $\Gamij{i}{j}{\overline{\matr{U}}}$ are separated by \Cref{eq-Dij-euivalence}.
Therefore,  there exists $\varepsilon > 0$ such that
\begin{equation*}
\frac{\lambda_2(\Gamij{i}{j}{\overline{\matr{U}}}) - \lambda_3(\Gamij{i}{j}{\overline{\matr{U}}})}{\lambda_1(\Gamij{i}{j}{\overline{\matr{U}}}) - \lambda_3(\Gamij{i}{j}{\overline{\matr{U}}})} < 1- \varepsilon.
\end{equation*}
By the  continuity of $\Gamij{i}{j}{\matr{U}}$ with respect to $\matr{U}$, the conditions of \Cref{lem:sufficient_decrease_Gamma} are satisfied in a neighborhood of $\overline{\matr{U}}$. Therefore, there exists $c>0$ such that 
\begin{align*}
|f(\matr{U}_{k}) - f(\matr{U}_{k-1})| &\ge c  \|\ProjGrad{h_k}{\uI_2}\| \|\matr{U}_{k}- \matr{U}_{k-1}\|\\
&\ge c\delta  \|\ProjGrad{f}{\matr{U}_{k-1}}\| \|\matr{U}_{k}- \matr{U}_{k-1}\|,
\end{align*} 
in a neighborhood of $\overline{\matr{U}}$  by \Cref{lem:sufficient_decrease_Gamma}, \Cref{lem:norm_equiv} and \eqref{inequality-gra-based}.  
By \Cref{rem:sufficient_descent_relax}, it is enough to use \Cref{theorem-SU15}, hence $\overline{\matr{U}}$ is the only limit point.
Moreover, by \Cref{lem:safeguard_Gamma}  and \eqref{inequality-gra-based},  the convergence rates apply. 

\item
Due to the scaling invariance,  $\overline{\matr{U}}$ belongs to an $n$-dimensional submanifold of stationary  points defined by $\overline{\matr{U}}\matr{S}$, where $\matr{S}$ is as in \eqref{eq:invariance_scaling}.
Since $\rank{\Hess{f}{\overline{\matr{U}}}} =  n(n-1)$, $f$ is Morse-Bott  at $\overline{\matr{U}}$ by \Cref{remMorseBott}.
Therefore, by \Cref{cor:PLMorseBottManifold},  $\zeta=1/2$ in \eqref{eq:Lojasiewicz}  at $\overline{\matr{U}}$, and thus  the convergence is linear by \Cref{theorem-SU15}.
\end{enumerate}
\end{proof}

\begin{theorem}\label{thm:local_linear}
Let  $f$ be as in \Cref{thm:accumulation_points}, and  ${\matr{U}}_*$ be a semi-strict local maximum  point  of $f$ (i.e.,  $\rank{\Hess{f}{{\matr{U}}_*}} =  n(n-1)$).
Then there exists a neighborhood $\set{W}$ of ${\matr{U}}_*$, such that for any  starting point  $\matr{U}_0 \in \set{W}$, \Cref{Jacobi-G-general} converges linearly to ${\matr{U}}_* \matr{S}$, where $\matr{S}$ is of the form \eqref{eq:invariance_scaling}.
\end{theorem}
\begin{proof}
Let ${\tUN{n}}$ be the quotient manifold defined in \cref{sec:quotient}. 
By \Cref{lem:f_invariant_Hessian}  we  have that $\rank{\Hess{\widetilde{f}}{\widetilde{\matr{U}_*}}} = \rank{\Hess{{f}}{{\matr{U}_*}}} = n(n-1)$, and therefore it is negative definite.
 Let us take the open neighborhood $\widetilde{\set{W}}$  of ${\widetilde{\matr{U}}_*}$   as in  \Cref{lem:morse_maximum}. 
For simplicity assume that $f(\matr{U}_*) = 0$.

Next, assume that $\widetilde{\matr{U}}_{k-1} \in \widetilde{\set{W}}$, and 
 consider the $\matr{U}_k = \matr{U}_{k-1} \Gmat{i_k}{j_k}{\matr{\Utwo}_k}$ with  $\matr{\Utwo}_k$  given as the maximizer of \eqref{eq-cost-quadratic-form}.
 Let  $b = f({\matr{U}}_{k-1})$. In what follows, we are going to prove that  $\widetilde{\matr{U}}_{k} \in\widetilde{\set{W}}_b$  (defined as in \Cref{lem:morse_maximum}), 
so that the sequence $\widetilde{\matr{U}}_k$ never leaves the set $\widetilde{\set{W}}$.

Recall that  $\matr{\Utwo}_k$  is computed as follows (see \cref{remark-eigenvectors}): take the vector $\vect{w}$  as in \eqref
{eq:w_definition}.
Take $\alpha_1 = -w_2 / \sqrt{1-w_1^2}$, $\alpha_2 = -w_3 / \sqrt{1-w_1^2}$ (we can assume $w_1 \neq 1$ because otherwise  $\matr{\Utwo}_k = \matr{I}_2$  and this case is trivial), and consider the following geodesic in $\UN{n}$:
\begin{equation*}
\gamma(t) =  \text{Exp}_{\matr{U}_{k-1}}\left(\matr{U}_{k-1} \mathcal{P}_{i,j}^{\T} (\matr{\Omega}t)\right) ={\matr{U}_{k-1}}  \Gmat{i_k}{j_k}{\exp(\matr{\Omega t})}, 
\end{equation*} 
where $\matr{\Omega} \in \Tang{{\matr{I}_2}} {{\tUN{2}}}$ is defined as in \eqref{eq:omega_geodesic}. 
The geodesic  starts at $\gamma(0) =  {\matr{U}_{k-1}}$, and  reaches  $\gamma(t_*) = \matr{U}_{k}$   at  $t_*  = \arccos(w_1) \in  (0, \frac{\pi}{2}]$ .
Note that  by \Cref{exp-map-quotient},  the corresponding curve $\widetilde{\gamma}$ is  a geodesic in the quotient manifold ${\tUN{n}}$.

Next, from \eqref{eq-cost-quadratic-form}   (applied to $\matr{\Gamma} =\Gamij{i_k}{j_k}{\matr{U}_{k-1}}$) we have that
\begin{equation*}
f(\gamma(t) )  = h(e^{\matr{\Omega t}})=  \begin{bmatrix} \cos t & -\alpha_1\sin t  & -\alpha_2\sin t \end{bmatrix} \matr{\Gamma} \begin{bmatrix} \cos t & -\alpha_1\sin t  & -\alpha_2\sin t \end{bmatrix}^{\T} + C,
\end{equation*}
hence  $f(\gamma(t) )$ can be represented (for some constants $A$, $C_1$)  as 
\begin{equation*}
f(\gamma(t)) = A\cos(2(t-t_*)) + C_1; 
\end{equation*}
 note that $A>0$ since $t_*>0$ is the maximizer. 

 Next,   by \Cref{rem:negative_deriv_geodesic}, we should have   $\frac{d}{dt^2}f(\gamma(0)) = -4A\cos(-2t_*) < 0$,  which implies   $\cos(2t_*) > 0$. Thus,  we can further reduce the domain where $t_*$ is located to   $t_* \in (0,\frac{\pi}{4}]$. 
Hence we have that $\frac{d}{dt}f(\gamma(t)) = -4A\sin(2(t-t_*)) > 0$ for  any $t\in [0,t_*)$,  and thus  the cost function is  increasing;  note that $\frac{d}{dt}f(\gamma(t_*)) = 0$ and there are no other stationary points in $t\in [0,t_*)$.

Next, by continuity and because $\widetilde{\set{W}}$ is open,  there exists  a small $\varepsilon > 0$ such that  $\widetilde{\gamma}(\varepsilon)$ is in the interior of $\widetilde{\set{W}}_{b}$.
By periodicity of $f(\gamma(t))$ and continuity, we have that there exists $t_2$ such that $\widetilde{\gamma}(t_2) \in \delta(\widetilde{\set{W}}_b)$ and $\widetilde{\gamma}(t) \in \widetilde{\set{W}}_{b}$ for all $t\in[0,t_2]$.
By Rolle's theorem, there exists a local maximum of $f(\gamma(t))$ in $[0,t_2]$.
Note that by construction, the closest positive local maximum to $0$ is at $t_*$. Therefore  $\widetilde{\matr{U}}_k =  \widetilde{\gamma}(t_*) \in \widetilde{\set{W}}_{b}$, hence we stay in the same neighborhood $\widetilde{\set{W}}$.

Finally, as a neighborhood of $\matr{U}_*\in \UN{n}$, we can take the preimage  $\set{W} = \pi^{-1}(\widetilde{\set{W}})$; also 
 linear convergence rate follows from \Cref{thm:accumulation_points}.
 The proof is complete. 
 
\end{proof}

\subsection{Examples of cost functions satisfying regularity conditions}\label{sec:examples}
In this subsection,  we provide examples when   the regularity conditions  of \Cref{thm:accumulation_points,thm:local_linear}  are satisfied for diagonalizable tensors and matrices  at the diagonalizing rotation. 
Recall that $\tens{A} \in \CC^{n \times \cdots \times n}$ is a diagonal tensor if all the elements are zero except  $\diag{\tens{A}}$.

\begin{proposition}\label{prop:diag_neg_def}
\begin{enumerate}[(i)]
\item For a set of jointly orthogonally  diagonalizable  matrices
\begin{equation*}
\matr{A}^{{(\ell)}} = \matr{U}_{*} \left[\begin{smallmatrix}\mu^{(\ell)}_{1} & &  0 \\  & \ddots &  \\ 0 & & \mu^{(\ell)}_{n}\end{smallmatrix}\right] \matr{U}^{\H}_{*},
\end{equation*}
such that for any pair $i \neq j$,
\begin{equation*}
\sum\limits_{\ell=1}^{L} (\mu^{(\ell)}_{i} -\mu^{(\ell)}_{j} )^2 > 0,
\end{equation*}
the matrix $\matr{U}_{*}$ is a semi-strict (as in \Cref{thm:local_linear}) local maximum  point  of the cost function \eqref{eq-cost-jade}.  
\item For an orthogonally diagonalizable $3$rd order tensor
\begin{equation*}
\tens{A} = \tens{D} \contr{1} \matr{U}_{*} \contr{2} \matr{U}^{*}_{*}\contr{3} \matr{U}^{*}_{*},
\end{equation*}
where $\tens{D}$ is a diagonal tensor with at most one zero element on the diagonal,
the matrix $\matr{U}_{*}$ is a semi-strict local maximum  point  of the cost function \eqref{eq-cost-3-tensor}. 
\item For an orthogonally diagonalizable  4th order  tensor
\begin{equation*}
\tens{A} = \tens{D} \contr{1} \matr{U}_{*} \contr{2} \matr{U}_{*} \contr{3} \matr{U}^{*}_{*}\contr{4} \matr{U}^{*}_{*},
\end{equation*}
where  $\tens{D}$ is a diagonal tensor,  the values on the diagonal are either (a) all positive or (b) there is at most one $i$ with $\tenselem{D}_{iiii} \le 0$, for which $\tenselem{D}_{iiii} + \tenselem{D}_{jjjj} > 0$ for all $j \neq i$,  
the matrix $\matr{U}_*$ is a semi-strict local maximum point  of
the function  \eqref{cost-4-order}. 
\end{enumerate}
\end{proposition}
For proving \Cref{prop:diag_neg_def}, we  need a lemma about Hessians of  multilinear forms.
\begin{lemma}\label{lem:Hessian_off}
Let $\gamma(\vect{u})$ be  a  Hermitian form  of order $2d$   $\gamma(\vect{u}) = g_{\tens{B},t}(\vect{u})$, where  $\tens{B}$ is  diagonal tensor.  
Then for any  distinct indices $1 \le i \neq j \neq k \le n$ it holds that 
\begin{equation*} 
\frac{\partial^2 {\gamma}}{\partial {u}^*_i \partial {u}^*_j}  (\vect{e}_{k} ) =  \frac{\partial^2 {\gamma}}{\partial {u}^*_i \partial {u}_j}  (\vect{e}_{k} ) = 0. 
\end{equation*}
\end{lemma}
\begin{proof}
By continuing  differentiation as in \Cref{lem-multilinear-form}, we get
that
\begin{equation*}
\vect{e}_{i}^{\T}\left(\frac{\partial^2 {\gamma}}{\partial \vect{u}^* \partial \vect{u}^*} (\vect{e}_{k})\right)\vect{e}_{j} = \sum\limits_{\begin{smallmatrix}s \neq  p \\ 1 \le s, p\le d\end{smallmatrix}} 
(\tens{B}  \contr{s} \vect{e_{i}} \contr{p} \vect{e_{j}})_{k\ldots k} = 0, 
\end{equation*}
\begin{equation*}
\vect{e}_{i}^{\T}\left(\frac{\partial^2 {\gamma}}{\partial \vect{u}^{*} \partial \vect{u}} (\vect{e}_{k})\right)\vect{e}_{j} = \sum\limits_{s=1}^{d}  \sum\limits_{p=d+1}^{2d} 
( \tens{B}  \contr{s} \vect{e_{i}} \contr{p} \vect{e_{j}})_{k\ldots k} = 0,
\end{equation*}
which completes the proof.
\end{proof}

\begin{proof}[Proof of \Cref{prop:diag_neg_def}]
Without loss of generality, we can consider only the case $\matr{U}_*=\matr{I}_n$, so that all the matrices/tensors are diagonal.
Due to diagonality of matrices/tensors (the off-diagonal elements are zero) from \Cref{prop:gradient-general-fn} we have that $\matr{I}_n$ is a stationary point 
and the Euclidean gradient $\nabla^{(\R)} f(\matr{I}_n)$ is  a diagonal matrix that contains $2d\diag{\tens{B}}$ on its diagonal.  
 Moreover,  by \cite[Eq. (8)--(10)]{Absil2013}  the Riemannian Hessian is  a sum of the projection of the Euclidean Hessian on the tangent space
 and a second term given by the Weingarten operator 
\begin{equation}\label{eq:Hessian_reduction}
\HessAppl{f}{\matr{I}_n}{\eta} = \Pi_{\mathbf{T}_{\matr{I}_n}\UN{n}}  {\rm H}_{{f}}(\matr{I}_n)  [\eta ]  +
\mathfrak{A}_{\matr{I}_n}(\eta, \Pi_{(\mathbf{T}_{\matr{I}_n}\UN{n})_{\bot}} \nabla^{(\R)} f(\matr{I}_n)), 
\end{equation}
where  $ {\rm H}_{{f}}$  is the Euclidean  Hessian  of  $f$,
 and the Weingarten operator  for $\UN{n}$ (similarly to the case of orthogonal group \cite{Absil2013}) is given by
\begin{equation*}
\mathfrak{A}_{\matr{U}}(\matr{Z},\matr{V})  =
\matr{U} \frac{1}{2} \left(\matr{Z}^{\H} \matr{V} - \matr{V}^{\H} \matr{Z}\right).
\end{equation*}
First,   we show that the  Euclidean  Hessian does not contain off-diagonal blocks. From \eqref{eq:Hessian_reduction}, we just need to look at the Euclidean Hessian.
Take two pairs of indices $(i,k)$ and $(j,l)$ and look at the  second-order Wirtinger derivatives  
\begin{equation*}
\frac{\partial^2 f}{\partial {U}^{*}_{i,k}\partial {U}_{j,l}} \quad \mbox{and} \quad \frac{\partial^2 f}{\partial {U}^{*}_{i,k}\partial {U}^{*}_{j,l}}.
\end{equation*}
Since by \eqref{eq:gradient_contrast_like}, $(\deriv{f}{\matr{U}^{*}})_{i,k}$ is a function of $\vect{u}_k$ only,  these  terms can only be nonzero if $j = k$ or $l = k$.
Let us choose $l=k$  (and $i\neq j$).  In that case, by \Cref{lem:Hessian_off}, 
\begin{equation*}
\frac{\partial^2 f}{\partial {U}^{*}_{i,k}\partial {U}_{j,k}} (\matr{I}_n) =  
\frac{\partial^2\gamma}{\partial {u}^{*}_i\partial{u}_j}(\vect{e}_k)= 0, 
\quad \frac{\partial^2 f}{\partial {U}^{*}_{i,k}\partial {U}^*_{j,k}} (\matr{I}_n) =  
\frac{\partial^2\gamma}{\partial {u}^{*}_i\partial{u}^*_j}(\vect{e}_k)= 0. 
\end{equation*}
Similarly,  we can show that off-diagonal blocks in   the second Hessian term is also equal to zero. 
Indeed, take $\matr{Z} = \mathcal{P}^{\T}_{i,k} (\matr{\Psi}_1)$, were $\matr{\Psi}_1$ is a $2\times 2$ skew-Hermitian matrix. 
Recall that $\matr{V} =  \Pi_{(\mathbf{T}_{\matr{I}_n}\UN{n})_{\bot}} \nabla^{(\R)} f(\matr{I}_n) =\nabla^{(\R)} f(\matr{I}_n)$ is diagonal, hence  $\matr{A} = \frac{\matr{Z}^{\H}\matr{V} -  \matr{V}^{\H}\matr{Z}}{2} = \mathcal{P}^{\T}_{i,k} (\matr{\Psi}_2)$ for some $2\times 2$ skew-Hermitian matrix $\matr{\Psi}_2$.
In this case, if $(j,l) \neq (i,k)$, then $\RInner{\matr{A}}{\mathcal{P}^{\T}_{j,l} (\matr{\Psi}_3)}= 0$ for any $2\times 2$ skew-Hermitian  $\matr{\Psi}_3$.
Thus the Riemannian Hessian is block-diagonal with the terms given in \Cref{prop:dir_deriv_Hessian}.

Finally, we apply \Cref{lem-cost-quadratic-form} and get that
\begin{enumerate}[(i)]
\item $\HessPS{i}{j}{\matr{I}_n}  = - \matr{I}_2 \sum\limits_{\ell=1}^{L} (\mu^{(\ell)}_{i} -\mu^{(\ell)}_{j} )^2$  for the cost function \eqref{eq-cost-jade}; 
\item $\HessPS{i}{j}{\matr{I}_n}  = - \frac{3}{2}\matr{I}_2 (|\tenselem{D}_{iii}|^2 + |\tenselem{D}_{jjj}|^2)$  for the cost function \eqref{eq-cost-3-tensor}; 
\item$\HessPS{i}{j}{\matr{I}_n}  = - \matr{I}_2 (\tenselem{D}_{iiii} + \tenselem{D}_{jjjj})$  for the cost function \eqref{cost-4-order}. 
\end{enumerate}
 It is easy to check that, in all three cases, $\HessPS{i}{j}{\matr{I}_n}$ is  negative definite  for any $i\neq j$  if and only if the conditions of the proposition are satisfied.
 The proof is complete. 
\end{proof}

\section{Implementation details  and experiments}
In this section, we comment on implementation details for \Cref{Jacobi-G-general} and Jacobi-type methods in general.
Note that implementations of Jacobi-type methods \cite{cardoso1993blind,cardoso1996jacobi,de1997signal,comon2001source} for the cyclic order of pairs are widely available, but they are often tailored to source separation problems and use implicit calculations.
The codes reproducing experiments in this section are publicly available at  \url{https://github.com/kdu/jacobi-G-unitary-matlab} (implemented in MATLAB, version R2019b).
Note that some experiments for the orthogonal group are available in \cite{LUC2018}.
\subsection{Implementation and computational complexity}\label{sec:impl_complexity}
Consider the general  problem of maximizing \eqref{cost-fn-general} (for $d \le 3$).
Note that the Givens rotations (from   \Cref{thm:elementary_updates}), as well as the Riemannian gradient (from \Cref{prop:gradient-general-fn}), are expressed in terms of the rotated tensors.
This leads to the following practical modification of \Cref{Jacobi-G-general}: instead of updating $\matr{U}_k = \matr{U}_{k-1} \Gmat{i_k}{j_k}{\matr{\Utwo}_k}$, we can transform the tensors themselves.
We summarize this idea in \Cref{alg:tensor_rotated} for the case $d=2$ (simultaneous diagonalization of matrices), and the cost function \eqref{eq-cost-jade}.

\begin{algorithm}[ht!]
{\bf Input:} Matrices $\matr{A}^{(\ell)}$, $1 \le \ell \le L$, starting point $\matr{U}_{0}$.\\
{\bf Output:} Sequence of iterations $\matr{U}_{k}$, rotated matrices $\matr{W}^{(\ell)}_k$.
\begin{enumerate}
\item initialize $\matr{W}^{(\ell)}_0  =   \matr{U}^{\H} \matr{A}^{(\ell)} \matr{U}$, for all $\ell$.
\item {\bf For} $k=1,2,\ldots$ until a stopping criterion is satisfied do
\item\quad Choose an index pair  $(i_k,j_k)$
\item\quad Find $\matr{\Psi}_k = \matr{\Psi}_k(c,s_1,s_2)$ that minimizes
\[
\textit{h}_k(\theta) = \sum\limits_{\ell} \|\diag{(\Gmat{i_k}{j_k}{\matr{\Psi}_k})^{\H} \matr{W}^{(\ell)}_{k-1}  \Gmat{i_k}{j_k}{\matr{\Psi}_k}}\|^2
\]       
\item\quad  Update $\matr{W}^{(\ell)}_{k} =  {(\Gmat{i_k}{j_k}{\matr{\Psi}_k})^{\H} \matr{W}^{(\ell)}_{k-1}  \Gmat{i_k}{j_k}{\matr{\Psi}_k}}$
\item {\bf End For}
\end{enumerate}
\caption{Jacobi-type algorithm by rotating the tensors}\label{alg:tensor_rotated}
\end{algorithm}

Let us comment  on the complexities of the steps (in what follows, we only count numbers of complex multiplications).
Some basic comments first:
\begin{itemize}
\item 
We can assume that the complexity of step 4 is constant $O(1)$: indeed, by  \Cref{thm:elementary_updates}, an eigenvector of a $3\times3$ matrix needs to be found.
\item In step 5, only a ``cross'' inside each of the matrices  is updated (the elements with the one of the indices $i$ or $j$.
This gives a total complexity (for naive implementation) of  $8Ln$ multiplications per update.
\end{itemize}
Thus, if a cyclic strategy \eqref{equation-Jacobi-C} is adopted  (the whole gradient is not computed),
then the  cycle of $\frac{n(n-1)}{2}$ plane rotations (often called \emph{sweep}) has the complexity $O(Ln^3)$.

\Cref{Jacobi-G-general} requires more care, since we need to have access to the Riemannian gradient (or the matrix $\matr{\Lambda}(\matr{U})$).
According to \Cref{prop:gradient-general-fn}, $O(Ln^2)$ multiplications are needed to compute the   matrix $\matr{\Lambda}(\matr{U})$ in the Riemannian gradient.
On the other hand, a plane rotation affects also only a part of the Riemannian gradient (also a cross) hence updating the   matrix $\matr{\Lambda}(\matr{U})$ after each rotation has complexity $O(Ln)$.
Thus, the complexity of  one sweep is again $O(Ln^3)$.

Now let us compare with the computational complexity of a first-order method from  \cite{Absil08:Optimization} (e.g., gradient descent).
At each iteration, we need  at least to compute the Riemannian gradient $O(Ln^2)$, and then compute the retraction, which has  complexity $O(n^3)$ for  typical choices (QR or polar decomposition).
Note that at each step we also need to  rotate the matrices, which requires additional $O(Ln^3)$ multiplications.

\begin{remark}
For 3rd order tensors, the complexity of the Jacobi-based methods does not increase, because we again update the cross, which has $O(Ln)$ elements.
\end{remark}

\subsection{Numerical experiments}
In the experiments, we again consider, for simplicity, simultaneous matrix diagonalization \eqref{eq-cost-jade}.
The general setup is as follows: we generate $L$  matrices $\matr{A}^{(\ell)} \in \CC^{n\times n}$, and compare several versions of 
\Cref{Jacobi-G-general}, as well as first-order Riemannian optimization methods implemented in the {\tt manopt} package \cite{BoumM14:manopt} (using {\tt stiefelcomplexfactory} ).
We compare the following methods:
\begin{enumerate}
\item \emph{Jacobi-G-max}: at each step of \Cref{Jacobi-G-general}, we select the pair $(i,j)$ that maximizes the absolute value $\Lambda_{i,j}(\matr{U}_{k-1})$ (see \Cref{rem:choosing_pair}).
\item \emph{Jacobi $0.1$}: we select the pairs in a cyclic-by-row order \eqref{equation-Jacobi-C}, but perform the rotations only if \eqref{inequality-gra-based} is satisfied for $\delta = 0.1 \sqrt{2}/n$. 
\item \emph{Jacobi-cyclic}: we use the cyclic-by-row order \eqref{equation-Jacobi-C}, without \eqref{inequality-gra-based}.
\item \emph{SD}: steepest descent from  \cite{BoumM14:manopt}.
\item \emph{CG}: conjugate gradients from  \cite{BoumM14:manopt}.
\item \emph{BFGS}: Riemannian version of BFGS from  \cite{BoumM14:manopt}.
\end{enumerate}
In all comparisons, $\matr{U}_0 = \matr{I}_n$. We also plot $\sum\limits_\ell \|\matr{A}^{(\ell)}\|^2 - f(\matr{U})$  instead of $f(\matr{U})$.

We first consider a difficult example. $L=5$ matrices of size $10\times 10$ were generated randomly, such that the real and imaginary part are sampled from the uniform distribution on $[0;1]$.
We plot the results in \Cref{fig:randA_comparison}.
\begin{figure}[htb!]
\begin{minipage}[b]{0.49\linewidth}
 \centering
 {\includegraphics[height=4cm]{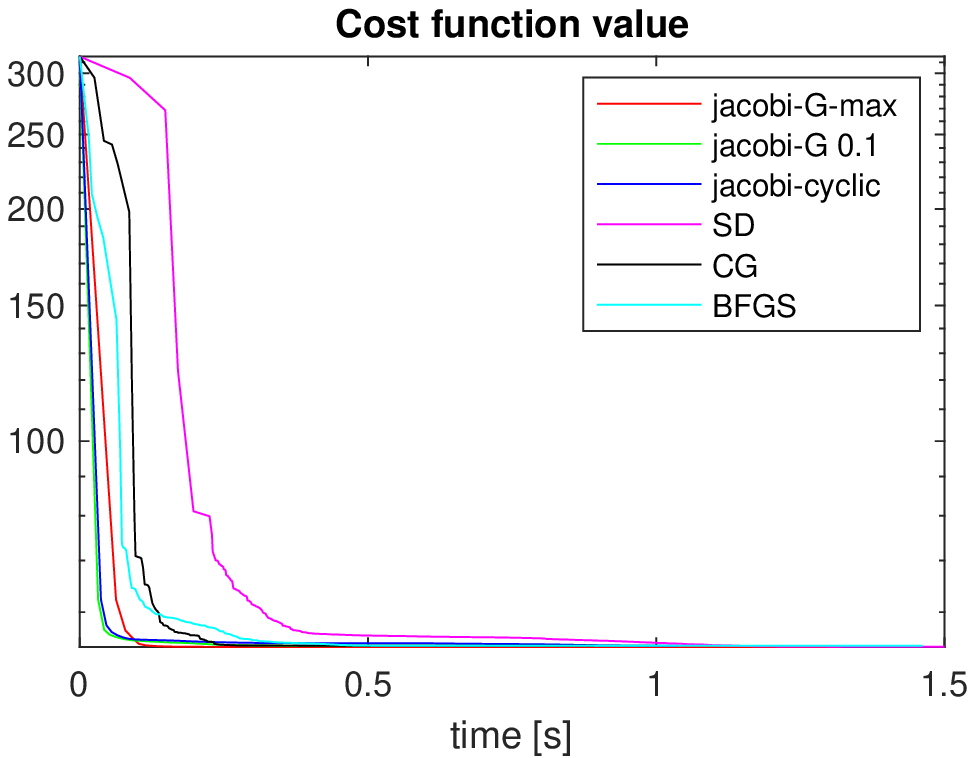}}
\end{minipage}
\hfill
\begin{minipage}[b]{0.49\linewidth}
 \centering
 {\includegraphics[height=4cm]{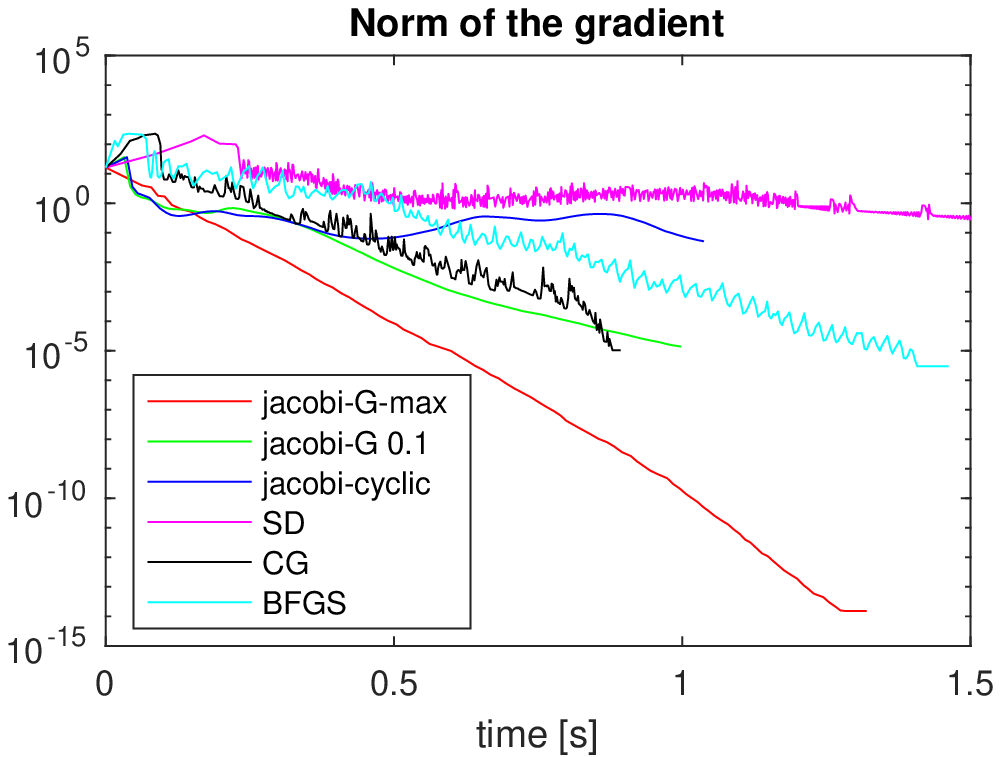}}
\end{minipage}
\caption{Cost function  value (left)  and norm of the gradient   (right), Example 1}
\label{fig:randA_comparison}
\end{figure}

We do not expect this example to be easy for all of methods: this example is far from a diagonalizable, and we are not likely to be in a small neighborhood of a local extremum.
We see that the Jacobi-type methods converge very fast, and for the versions of \Cref{Jacobi-G-general} the gradient seems to converge to zero. 
We also see that the Jacobi-G-max version is the best compared to Jacobi-G with cyclic order and fixed $\delta$ (we tried different values of $\delta$).

We also consider a nearly diagonalizable case, $n=L = 20$.
We take $\matr{A}^{(\ell)} = \matr{Q}^{\H} \matr{D}^{(\ell)} \matr{Q} + \matr{E}^{(\ell)}$, where $\matr{Q}$ is a random unitary matrix, elements of $\matr{E}^{(\ell)}$ are i.i.d. realizations of Gaussian random variable with standard deviation $10^{-6}$, and $\matr{D}^{(\ell)}$ is a diagonal matrix, whose diagonal elements are equal to $1$, except the element ${D}^{(\ell)}_{\ell,\ell} = 2$ (note that such matrices, without noise, satisfy \Cref{prop:diag_neg_def}).
\begin{figure}[htb!]
\begin{minipage}[b]{0.49\linewidth}
 \centering
 {\includegraphics[height=4cm]{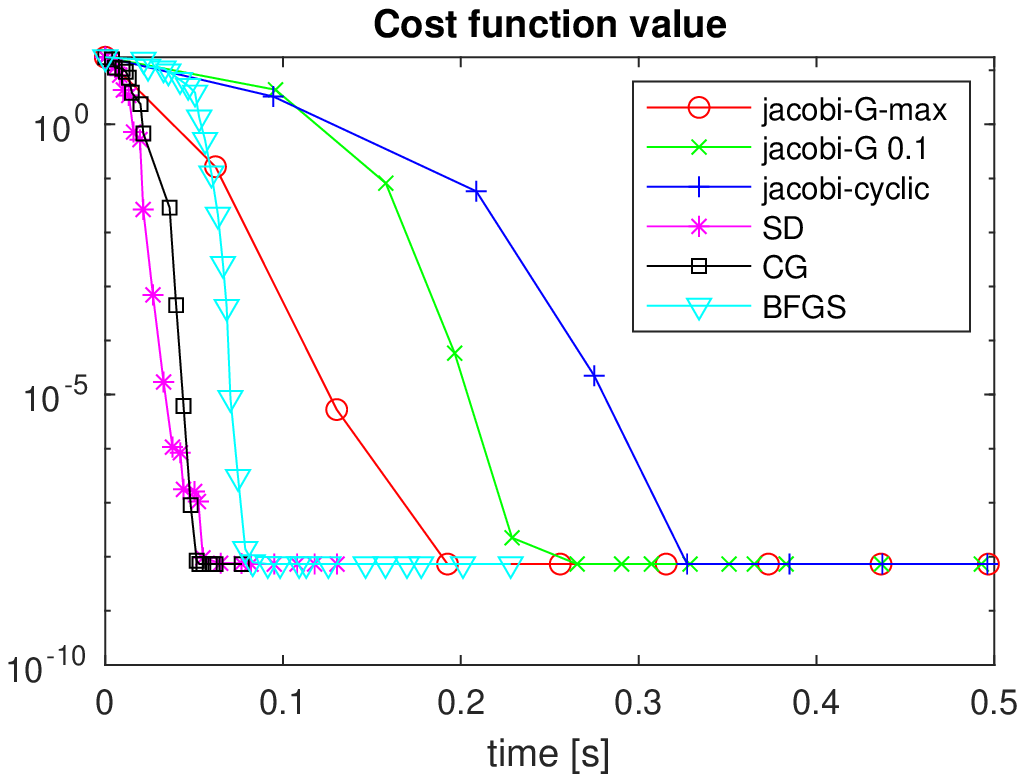}}
\end{minipage}
\hfill
\begin{minipage}[b]{0.49\linewidth}
 \centering
 {\includegraphics[height=4cm]{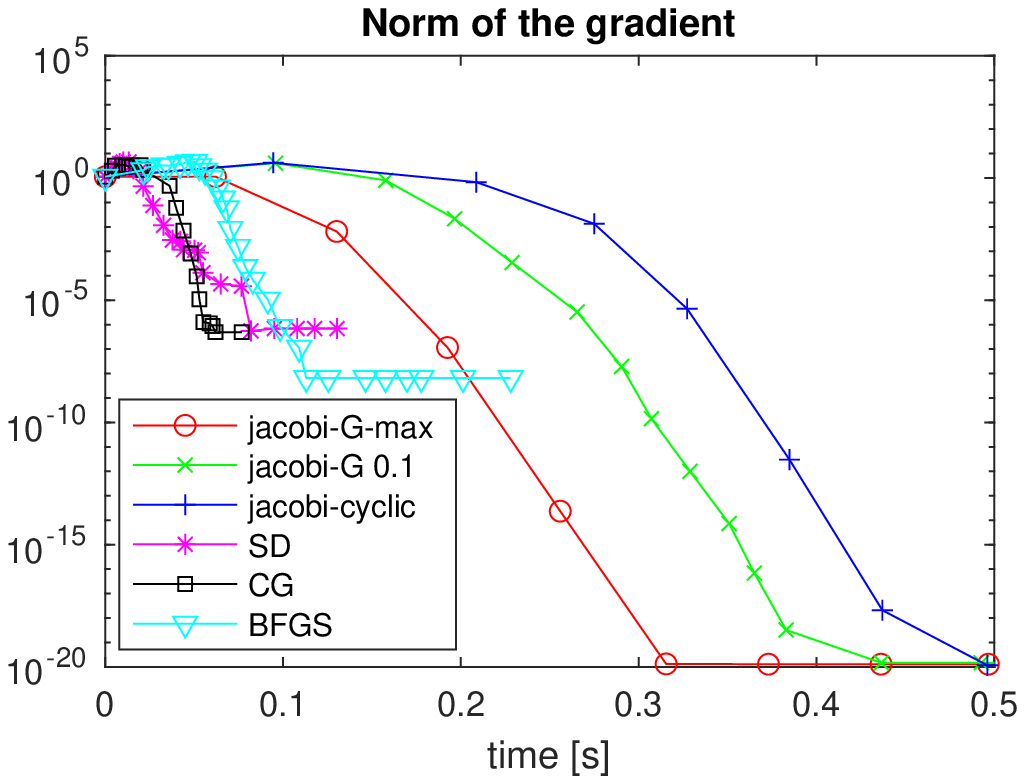}}
\end{minipage}
\caption{Cost function  value (left) and norm of the gradient  (right),  Example 2}
\label{fig:diag_comparison}
\end{figure}

We plot the results in \Cref{fig:diag_comparison}. 
We see that the convergence of  general-purpose Riemannian algorithms is much better in this case.
Still, Jacobi algorithms converge in a  few sweeps. 

Note that in the current implementation (used to produce \Cref{fig:randA_comparison} and \Cref{fig:diag_comparison}), we do not use the $O(Ln)$ update of $\matr{\Lambda}(\matr{U})$ as suggested in \cref{sec:impl_complexity} (i.e., the matrix  $\matr{\Lambda}(\matr{U})$ is recalculated at each step).
This can be observed in \Cref{fig:diag_comparison}, where each marker for the Jacobi-type methods corresponds to one sweep.
Thus, a further speedup of Jacobi-type methods is possible.

{
\section{Discussion}
In this paper, we showed that for a class of optimization problems on the unitary group (corresponding to approximate matrix and tensor diagonalization), convergence of Jacobi-type algorithms to stationary points can be proved (together with  convergence rates).
A gradient-based order of Givens rotations is adopted (which extends the approach of \cite{IshtAV13:simax} for the real-valued case).
By using the tools based on \L{}ojasiewicz gradient inequality, we can ensure single-point convergence, under regularity conditions on one of  the accumulation points;
the speed of convergence is linear for the  non-degenerate case, and local convergence can be proved.
We also provided a characterization of Jacobi rotations for tensors of arbitrary orders.

Still, we believe that stronger results can be obtained.
For the matrix case,  although  the Jacobi-type algorithms are similar in spirit to block-coordinate descent, they enjoy quadratic convergence  (of the cost function value)   for the classic matrix case \cite{GoluV96:jhu}  and the case of a pair of commuting matrices \cite{BunsBM93:simax}. 

Also, in the matrix case, many results are available for cyclic strategies (at least weak convergence is known, see \cite{GoluV96:jhu}).
It would be interesting to see if similar results can be proved for  tensor and joint matrix diagonalization cases;
in fact, the convergence for the pure cyclic strategy is often observed in practice (see \cite{LUC2018} for a comparison in the case of orthogonal group), but there is no convergence proof.

Note that we were not able to prove global single-point convergence, as in \cite{LUC2018} (proved for 3rd order tensors or matrices).
It seems that in the complex case, not only the order of rotations matters (which makes it similar to the higher-order case \cite{LUC2018}).
One possible track is to modify of a way to find the Jacobi rotation itself (i.e. adopt proximal-like steps if needed, see also \cite{LUC2018}).

Another interesting question is whether we can relax the definition of single-point convergence. 
Indeed, if the critical point is degenerate (even in the quotient manifold), then a natural question is whether the potentially different accumulation points, belong the same critical manifold.
This is, in fact, what is typically proved\footnote{In fact, it seems that \cite{LUC2018} is the first paper explicitly showing single-point convergence for the single real-valued matrix case, as the result of \cite{LUC2018}  also apply to the eigenvalue problems.} for the matrix case \cite{Drma10:simax}: if there are multiple eigenvalues, then the convergence of invariant subspaces is guaranteed (which corresponds to the same critical manifold).

\appendix
\section{Multilinear algebra proofs}\label{sec:ml_proofs}

\begin{proof}[Proof of \Cref{prop:equivalence_cost_functions}]
The ``only if'' part follows from the fact that if $\tens{B}$ is Hermitian,  there exist tensors $\tens{A}^{(1)},\ldots,\tens{A}^{(L)}$ of order $d$ and real numbers $\alpha_\ell$, such that
\begin{equation}\label{eq:hermitian_spectral}
\tens{B} = \sum\limits_{\ell=1}^{L}  \alpha_\ell (\tens{A}^{(\ell)})^* \otimes \tens{A}^{(\ell)}, \text{ i.e. }\tenselem{B}_{i_1\ldots i_dj_1\ldots j_d} = \sum\limits_{\ell=1}^{L}  \alpha_\ell  (\tenselem{A}^{(\ell)}_{i_1\ldots i_d})^*  \tenselem{A}^{(\ell)}_{j_1\ldots j_d}.
\end{equation}
This  is nothing but the spectral theorem for Hermitian matrices applied to a matricization  of $\tens{B}$, see also \cite[Propositions 3.5 and 3.9]{JianL16:characterizing}.
Then \eqref{eq:hermitian_spectral} implies that 
\[
g_{\tens{B}, d}(\vect{u}) = \sum\limits_{\ell=1}^L  \alpha_\ell  | g_{\tens{A}^{(\ell)}, 0}(\vect{u}) |^2.
\]
In order to prove the ``if'' part, we make the following remarks.
\begin{enumerate}[(a)]
\item When restricted to $\|\vect{u}\| = 1$, the order of the form can be always increased; 
Indeed, suppose that $\tens{T}$ is $2(d-1)$-order Hermitian, then for all $\|\vect{u}\| = 1$
\[
g_{\tens{T}, d-1}(\vect{u}) = 
(\tens{T} \otimes \matr{I}_n) \contr{1} \vect{u}^{*} \cdots \contr{d-1} \vect{u}^{*} \contr{d} \vect{u} \cdots  \contr{2(d-1)} \vect{u}\contr{2d-1} \vect{u}^* \contr{2d} \vect{u},
\]
where $\matr{I}_n$ is the identity; the expression on the right-hand side is a Hermitian form \eqref{eq-hermitian-form}, where the   $2d$-order tensor $\tens{B}$ can be defined by permuting the indices:
\[
\tenselem{B}_{i_1 \ldots i_d j_1 \ldots j_d} =(\tens{T} \otimes \matr{I}_n)_{i_1\ldots i_{d-1}j_1\ldots j_{d-1}i_dj_d} =  \tenselem{T}_{i_1\ldots i_{d-1}j_1\ldots j_{d-1}} (\matr{I}_n)_{i_dj_d}.
\] 
\item Note that for any $t$, the function $|g_{\tens{A}, t}(\vect{u})|^2 =  g_{\tens{A}, t}(\vect{u})g^*_{\tens{A}, t}(\vect{u})$ is also a $2d$-form: 
\begin{equation}\label{eq-g-A-squared}
|g_{\tens{A}, t}(\vect{u})|^{2} =  (\tens{A} \otimes \tens{A}^*)\contr{1} \vect{u}^{*} \cdots \contr{ t} \vect{u}^{*} \contr{ t+1} \vect{u} \cdots   \contr{d+ t} \vect{u} \contr{d+ t+1} \vect{u}^* \cdots  \contr{2d} \vect{u}^*,  
\end{equation}
which can be written\footnote{An alternative shorter proof of  part (b) follows from the fact that  a $2d$-order form $g_{\tens{B}, d}(\vect{u})$ is real-valued if and only if it is Hermitian, see \cite[Proposition 3.6]{JianL16:characterizing}} as  $g_{\tens{B}, d}(\vect{u})$ for a tensor $\tens{B}$ obtained by permuting  indices:
\[
\tenselem{B}_{i_1 \ldots i_d j_1 \ldots j_d} =
(\tens{A} \otimes \tens{A}^*)_{i_1\ldots i_t j_{t+1} \ldots j_d j_{1}\ldots j_{t} i_{t+1} \ldots i_d} = \tenselem{A}_{i_1\ldots i_t j_{t+1} \ldots j_d}\tenselem{A}^*_{ j_{1}\ldots j_{t} i_{t+1} \ldots i_d}.
\]
Finally, sums of Hermitian tensors are Hermitian, which completes the proof.
\end{enumerate}
\end{proof}

\begin{proof}[{Proof of \Cref{thm:elementary_updates}}]
Since the cost function has the form \eqref{eq-sum-simple-functions}, we have
\[
\widetilde{h}(c,s_1,s_2) = g_{\tens{T},d}(\left[\begin{smallmatrix} c\\ s^{\ast}\end{smallmatrix}\right]) +g_{\tens{T},d}(\left[\begin{smallmatrix} -s\\ c\end{smallmatrix}\right]) 
\]
Let us rewrite the first term using the double contraction:
\begin{align}
g_{\tens{T},d}(\left[\begin{smallmatrix} c\\ s^{\ast}\end{smallmatrix}\right]) & = 
\tens{T} \contr{1,d+1} \left(\begin{bmatrix} c\\ s^{*}\end{bmatrix}^{*} \begin{bmatrix} c& s^{*} \end{bmatrix} \right) \cdots \contr{d,2d} \left(\begin{bmatrix} c\\ s^{*}\end{bmatrix}^{*} \begin{bmatrix} c& s^{*} \end{bmatrix} \right) \notag \\
&= 
 \frac{1}{2^d}
\tens{T} \contr{1,d+1} (\matr{I}_2  + \matr{R}) \contr{2,d+2} (\matr{I}_2 +\matr{R}) \cdots \contr{d,2d} (\matr{I}_2 +\matr{R}),\label{eq:2x2_first}
\end{align}
\[
 \text{where }\matr{R} \eqdef \begin{bmatrix} 2c^2-1& 2cs^{\ast}\\2cs & 1-2c^2\end{bmatrix}, \quad \text{so that }
 \begin{bmatrix} c\\ s^{*}\end{bmatrix}^{*} \begin{bmatrix} c& s^{*} \end{bmatrix}= 
\begin{bmatrix} c^2& cs^{\ast}\\cs & |s|^2\end{bmatrix}
= \frac{1}{2} ( \matr{I}_2 + \matr{R}).
\]
Similarly, by noting that
\[
\begin{bmatrix} -s\\ c\end{bmatrix}^{*} \begin{bmatrix} -s& c \end{bmatrix}= 
\begin{bmatrix} |s|^2& -cs^{*}\\-cs & c^2\end{bmatrix}
= \frac{1}{2} ( \matr{I}_2 -\matr{R}),
\]
we can rewrite the second term
\begin{equation}\label{eq:2x2_second}
 g_{\tens{T},d}(\left[\begin{smallmatrix} -s\\ c\end{smallmatrix}\right]) = 
 \frac{1}{2^d}
\tens{T} \contr{1,d+1} (\matr{I}_2  - \matr{R}) \contr{2,d+2} (\matr{I}_2 -\matr{R}) \cdots \contr{d,2d} (\matr{I}_2 -\matr{R}).
\end{equation}
When summing \eqref{eq:2x2_first} and \eqref{eq:2x2_second}, we note that the odd powers of $\matr{R}$ cancel, and  the even powers have positive signs; therefore, due to symmetries we get
\begin{equation*}
\widetilde{h}(c,s_1,s_2) = \frac{1}{2^{d-1}} \sum\limits_{j=0}^{m} 
{d \choose {2j}} \tens{T} \contr{1,d+1} \matr{R} \cdots \contr{2j,d+2j}  \matr{R}   \contr{2j+1,d+2j+1} \matr{I}_2 \cdots \contr{d,2d}  \matr{I}_2,
\end{equation*}
where the binomial coefficient appears  when we sum over all possible locations of $\matr{R}$.

Next, we remark that $\matr{R}$ can be expressed in the following orthogonal basis 
\begin{equation}\label{eq:basis_matr_R}
\matr{R} = r_1 \bsmallmatrix{1 & 0\\0&-1} + r_2 \bsmallmatrix{0 & -1\\-1&0} + r_3   \bsmallmatrix{0 & i\\-i&0},
\end{equation}
where $\vect{r} = \begin{bmatrix} r_1 &  r_2 & r_3\end{bmatrix}^{\T}  = \begin{bmatrix} 2c^2-1 &  -2cs_1 & -2cs_2\end{bmatrix}^{\T}$ is defined in \eqref{eq:w_definition}.
Then by the multilinearity of the contractions, we can rewrite the expression for $\widetilde{h}(c,s_1,s_2)$ as
\[
\widetilde{h}(c,s_1,s_2) =  \sum\limits_{j=0}^{m} \tens{F}^{(j)} \contr{1}  \vect{r} \cdots  \contr{2j}  \vect{r},
\]
where each $\tens{F}^{(j)}$ is a symmetric complex $2j$-order  $3\times \cdots \times 3$ tensor, whose entries are obtained by contractions  of $\tens{T}$ with basis matrices in \eqref{eq:basis_matr_R} or  $\matr{I}_2$.

It is only left to show that all the elements in each of the tensors $\tens{F}^{(j)}$ are real. 
This is indeed the case, because for a Hermitian tensor $\tens{T}$ contraction with 
one of  the basis  matrices keeps it  Hermitian:
\begin{align*}
& (\tens{T} \contr{1,d+1} \bsmallmatrix{0 & i\\-i&0})_{i_2\ldots i_d, j_2 \ldots j_d} =
-i(\tenselem{T}_{2i_2\ldots i_d,1 j_2 \ldots j_d} - \tenselem{T}_{1i_2\ldots i_d,2 j_2 \ldots j_d}) \\
&=  i(\tenselem{T}^{*}_{2 j_2 \ldots j_d,1i_2\ldots i_d } - \tenselem{T}^{*}_{1 j_2 \ldots j_d,2i_2\ldots i_d} )  = 
 ((\tens{T} \contr{1,d+1} \bsmallmatrix{0 & i\\-i&0} )_{j_2 \ldots j_d,i_2\ldots i_d})^*,
\end{align*}
and similarly for contractions with $\bsmallmatrix{1 & 0\\0&-1}$, $\bsmallmatrix{0 & -1\\-1&0}$ and $\matr{I}_2$.

Finally, we note that, since $\|\vect{r}\| =1$, all the tensors $\tens{F}^{(j)}$ can be combined in one tensor $\tens{F}$ of order $2m$, as in  the proof of \Cref{prop:equivalence_cost_functions} (see part (a) of the proof).
\end{proof}

\section*{Acknowledgments} 
The authors would like to acknowledge  the two anonymous reviewers and the associate editor for their useful remarks that helped to improve the presentation of the results.

\nocite{edelman1998geometry}
\nocite{helgason1978}

\bibliographystyle{siamplain}
\bibliography{Jacobi_G_Unitary_Group}

\end{document}